\documentclass[a4,10]{amsart}
\usepackage{amssymb,amsmath,latexsym,amscd,amsfonts}
\usepackage{mathrsfs}

\usepackage[colorlinks=true, linkcolor=black]{hyperref}
\usepackage{cleveref}
\usepackage{mathtools}
\usepackage{graphicx}
\usepackage[all]{xy}
\usepackage{color}

\usepackage[margin=3.55cm]{geometry}

\usepackage{ulem}
\usepackage[shortlabels]{enumitem}
\allowdisplaybreaks[4]

\numberwithin{equation}{section}
\newtheorem{theorem}{Theorem}[section]

\newtheorem{proposition}[theorem]{Proposition}
\newtheorem{corollary}[theorem]{Corollary}
\newtheorem{lemma}[theorem]{Lemma}

\theoremstyle{definition}
\newtheorem{definition}[theorem]{Definition}

\newtheorem{remark}[theorem]{Remark}
\newtheorem{notation}[theorem]{Notation}
\newtheorem{example}[theorem]{Example}

\DeclareMathOperator*{\dprime}{\prime \prime}

\newcommand{\mn}{\mathcal{N}}
\newcommand{\mpp}{\mathcal{P}}
\newcommand{\mz}{\mathcal{Z}}
\newcommand{\mm}{\mathcal{M}}
\newcommand{\mq}{\mathcal{Q}}
\newcommand{\mh}{\mathcal{H}}

\newcommand{\tr}{\mathrm{tr}}

\newcommand\numberthis{\addtocounter{equation}{1}\tag{\theequation}}

\newcommand{\mr}{\mathcal{R}}

\newcommand{\C}{\mathbb{C}}
\newcommand{\rar}{\rightarrow}
\newcommand{\la}{\langle}
\newcommand{\ra}{\rangle}

\begin{document}
\title{Lattice of intermediate subalgebras}
\author[K C Bakshi]{Keshab Chandra Bakshi}
\address{Chennai Mathematical Institute, Chennai, INDIA}
\email{bakshi209@gmail.com,kcbakshi@cmi.ac.in}

\author[V P Gupta]{Ved Prakash Gupta} \address{School of Physical
  Sciences, Jawaharlal Nehru University, New Delhi, INDIA}
\email{vedgupta@mail.jnu.ac.in}

\date{May 03, 2020}
\begin{abstract}
Analogous to subfactor theory, employing Watatani's notions of index
and $C^*$-basic construction of certain inclusions of
$C^*$-algebras, (a) we develop a Fourier theory (consisting of Fourier
transforms, rotation maps and shift operators) on the relative
commutants of any inclusion of simple unital $C^*$-algebras with
finite Watatani index, and (b) we introduce the notions of interior
and exterior angles between intermediate $C^*$-subalgebras of any
inclusion of unital $C^*$-algebras admitting a finite index conditional
expectation. Then, on the lines of \cite{BDLR}, we apply these
concepts to obtain a bound for the cardinality of the lattice of
intermediate $C^*$-subalgebras of any irreducible inclusion as in (a),
and improve Longo's bound for the cardinality of intermediate
subfactors of an inclusion of type $III$ factors with
finite index. Moreover, we also show that for a fairly large class of
inclusions of finite von Neumann algebras, the lattice of intermediate
von Neumann subalgebras is always finite.
\end{abstract}

\dedicatory{Dedicated to V.~S.~Sunder}

\thanks{The first named author was supported partially by a
  postdoctoral fellowship of the National Board of Higher Mathematics
  (NBHM), India.}
  \maketitle
\tableofcontents

\section{Introduction}
Among the various significant themes of operators algebras, the theory
of {\it subfactors} has attracted a fair share of limelight during the
last three and a half decades because of the deep relationship and
implications it has exhibited to various other branches of Mathematics
as well as Theoretical Physics. The modern theory of {\it subfactors}
was initiated by Vaughan Jones in 1983 in his seminal work \cite{Jo},
wherein, among various deep and original ideas, he formalized the
notion of the index $[M:N]$ as the Murray-von Neumann's coupling
constant $\mathrm{dim}_NL^2(M)$, for any subfactor $N \subset M$ of
type $II_1$, and introduced the notion of the {\it basic construction}
for any unital inclusion of finite von Neumann algebras. Later, Kosaki
\cite{kosaki} generalized the notion of index and basic construction
in terms of suitable conditional expectations for subfactors of any
type. And, in 1990, Watatani \cite{Wa} generalized Jones' and Kosaki's
indices to the index of a conditional expectation associated to any
inclusion of $C^*$-algebras. In the same article, using the language
of Hilbert $C^*$-modules, Watatani also provided an analogue of their
notions of basic construction for any pair $B \subset A$ of unital
$C^*$-algebras with respect to  a {\it finite index} conditional expectation. Over
the years, many authors have used Watatani's notions of index and
$C^*$-basic construction to prove significant results in the theories
of $C^*$-algebras, von Neumann algebras and Hilbert $C^*$-modules -
see \cite{KaWa, I2, I3, KW, KPW1, Po4, I,  FL,  KPW2, IW}.

Since the basic flavour of the theory of subfactors revolves around
the analysis of the relative position of a subfactor inside an ambient
factor, it is a very natural and fundamental question to analyze the
{\it lattice} consisting of all intermediate subfactors. Needless to
mention, a substantial amount of work has been done in this direction
too. For instance, Bisch \cite{Bi} exhibited a dictionary between the
intermediate subfactors of a subfactor $N \subset M$ of type $II_1$
and the so-called {\it biprojections} in the relative commutant space
$N'\cap M_1$. See \cite{BJ} for some interesting results in this
direction.  The crucial ingredient in Bisch's {\it biprojection}
theory is the {\it Fourier theory} on the relative commutants
$N^{\prime}\cap M_k$ formulated by Ocneanu and Jones - see \cite{Oc,
  Bi,Jo2}. Furthermore,  subfactor theory has gained a
lot from the structures of Popa's {\it $\lambda$-lattice} (\cite{Po2}) and
Jones' {\it planar algebra} (\cite{Jo2}) on the {\it standard invariant} of
any subfactor of type $II_1$ with finite Jones index, both of which
were formulated by exploiting the techniques of {\it Fourier theory} quite heavily (see
\cite{Po1} for details).

On the other hand, the study of $C^*$-subalgebras of a given
$C^*$-algebra has also attracted good attention and that too from
different perspectives - see \cite{Ch1,Ch2, I, IW, Su} and the
references therein. In Section 2, after a quick recollection of
Watatani's notions of index and $C^*$-basic construction, and the
notion of {\it minimal conditional expectations} by Hiai, Kosaki and
Longo (\cite{H, KL, Lon1, Lon2, KaWa, Wa}), given any inclusion $B
\subset A$ of simple unital $C^*$-algebra with finite Watatani index,
we single out a sequence of consistent tracial states on the tower of
relative commutants, which then allows us to obtain a bound for the
dimension of each relative commutant $B'\cap A_k$.

Then, in Section 3, we provide a $C^*$-version of the Fourier theory
for any such pair of simple unital $C^*$-algebras.  The subtle
difference between our approach and that of Ocneanu and Jones
lies in the fact that, unlike for finite factors, we neither have a
tracial state on $A$ to begin with nor the `modular conjugation
operator' on the $L^2$-completion of $A$. As mentioned above, we found a
way around using the notion of {\it minimal conditional expectations}.
We provide a detailed theory of {\it Fourier transforms, rotation maps}
and {\it shift operators} on the {\it relative commutants} of
appropriate inclusions of $C^*$-algebras. 

In Section 4,  motivated by Bisch's characterization of intermediate
subfactors in terms of {\it biprojections}, we formulate the notions of
{\it biunitaries}, {\it bipartial isometries} and {\it biprojections} and their
behaviour under Fourier transforms and rotations. As the first
application of the $C^*$-Fourier theory, given any irreducible pair $B
\subset A$ of simple unital $C^*$-algebras with finite Watatani index
and a biprojection in $B'\cap A_1$, we provide a recipe to obtain an
intermediate $C^*$-subalgebra of the dual pair $A \subset A_1$ in
\Cref{bischprojection}.

Recently, the first named author along with Das, Liu and Ren, in
\cite{BDLR}, introduced the notions of interior and exterior angles
between intermediate subfactors of a subfactor of type $II_1$ to
understand the relative position of two intermediate
subfactors. Motivated by them, in Section 5, we begin with the
introduction of the notions of interior and exterior angles between
any two intermediate $C^*$-subalgebras of an inclusion $B \subset A$ of
unital $C^*$-algebras with a finite index conditional expectation, and
provide some useful expressions for the same.  Then, very much like the
minimal intermediate subfactors of a subfactor of type $II_1$ (as was
exhibited in \cite{BDLR}), we show in \Cref{m1} that, in terms of relative
positions, there is a certain rigidity observed by the minimal
intermediate $C^*$-subalgebras of an irreducible pair of simple
$C^*$-algebras in the sense that the interior angle between any two
such $C^*$-subalgebras is always greater than
$\pi/3$. The proof is based on the $C^*$-Fourier theory that we
develop.

On the other hand, Watatani in \cite{Wa2} (see also \cite{Po3}) and
then Teruya and Watatani in \cite{TW} showed that the lattice of
intermediate subfactors of an irreducible subfactor of type $II_1$ and
type $III$, respectively, is finite.  Then, Longo (in \cite{Lon})
proved that the number of intermediate subfactors of an irreducible
subfactor $N \subset M$ (of any type) with finite index is bounded by
$([M:N]^2)^{[M:N]^2}$ and had asked whether the bound could be
improved to $[M:N]^{[M:N]}$. The authors of \cite{BDLR} exploited the
notion of interior angle satisfactorily to answer this question and
showed that for an irreducible subfactor $N \subset M$ of type $II_1$
the bound can be improved significantly to $\min\{9^{[M:N]},
[M:N]^{[M:N]}\}$. However, the question for irreducible subfactors of
type $III$ remained unanswered. For $C^*$-algebras, Ino and Watatani
(in \cite[Corollary 3.9]{IW}) had shown that every irreducible pair $B
\subset A$ of simple unital $C^*$-algebras with a conditional
expectation of finite index has only finitely many intermediate
$C^*$-subalgebras. However, they did not provide any bound on the
number of such intermediate $C^*$-subalgebras.

As another useful application of the $C^*$-Fourier theory and the
notion of interior angle, on the lines of \cite{BDLR}, we
deduce (in \Cref{bound-thm-1}) that the number of intermediate
$C^*$-subalgebras of an irreducible pair $B \subset A$ of simple
unital $C^*$-algebras with finite Watatani index is bounded by $\min
\big\{9^{{[A:B]}_0^2},{\big({[A:B]}_0^2}\big)^{{[A:B]}_0^2}\big\}$,
where $[A:B]_0$ denotes the Watatani index of the pair $B \subset
A$. As was observed in \cite{BDLR}, the essence of this
  proof lies in the above mentioned rigidity phenomenon observed by
  the minimal intermediate $C^*$-subalgebras, which then allows one to
  deduce that the number of such intermediate subalgebras is bounded
  by the Kissing number ${\tau}_n$ of the $n$-dimensional sphere,
  where $n=\text{dim}_{\C}(B^{\prime}\cap A_1)$. The same tools allow
us to deduce (in \Cref{bound-thm-2}) that the improved bound obtained
in \cite{BDLR} holds even for the lattice of intermediate subfactors
of an irreducible $\sigma$-finite subfactor of type $III$ of finite
index, and thereby answers the question of Longo (\cite{Lon}) for the
type $III$ case as well.\smallskip

Finally, in the last section, using Christensen's perturbation
technique from \cite{Ch1} and Watatani's
compactness argument from \cite{Wa2}, we generalize the above
mentioned finiteness results of Watatani and Teruya by showing (in
\Cref{main}) that the lattice of intermediate von Neumann subalgebras
of an unital inclusion $\mn\subset \mm$ is finite if $\mm$ is a finite
von Neumann algebra with a normal tracial state $\tr$ on $\mm$ such
that the unique $\tr$-preserving conditional $E^{\mm}_{\mn}: \mm \rar
\mn$ has finite Watatani index, $\mz(\mn)$ is finite dimensional and
${\mn}^{\prime}\cap \mm$ equals either $\mz(\mn) $ or $\mz(\mm)$. We
conclude the paper with some nice corollaries.

\section{Inclusions of simple unital $C^*$-algebras}
Generalizing the notions of indices and basic constructions by Jones
\cite{Jo} and Kosaki \cite{kosaki}, Watatani, in \cite{Wa}, proposed the notion of
 a vector-valued index for  conditional expectations of inclusions of
$C^*$-algebras and the notion of basic construction of such inclusions. In this section,
we briefly recall the two notions and present some consequences which
will be used later and are of independent interest as well.

\subsection{Watatani index of conditional expectations}\( \)

 Given a pair $B \subset A$ of unital $C^*$-algebras (with a common
 identity), a conditional expectation $E:A \rightarrow B$ is a
 positive projection with norm one such that $E(axb)=aE(x)b$ for all
 $a, b\in B $ and $x\in A$. A conditional expectation $E:A \rightarrow
 B$ is said to have \textit{finite index} if there exists a finite set
 $\{\lambda_1,\ldots,\lambda_n\}\subset A$ such that $x=\sum_{i=1}^n
 E(x\lambda_i)\lambda^*_i=\sum_{i=1}^n \lambda_iE(\lambda^*_ix)$ for
 every $x\in A$. Such a set $\{\lambda_1,\cdots,\lambda_n\}$ is called
 a \textit{quasi-basis} for $E$. This is a generalization of the
 notion of Pimsner-Popa basis (\cite{PiPo1}) for a pair of von Neumann
 algebras with a conditional expectation. The Watatani index of $E$ is
 given by
 \[
 \mathrm{Ind} (E)= \sum_{i=1}^n \lambda_i\lambda^*_i, 
 \]
and is independent of the quasi-basis. Let $\mathcal{E}_0(A,B)$ denote
the set of all conditional expectations from $A$ onto $B$ of finite
index. 

In general, $\mathrm{Ind} (E)$ is not a scalar but it is an invertible
positive element of $\mz(A)$. Motivated by the values of Jones index
for subfactors (\cite{Jo}), Watatani showed the following:
\begin{theorem}\label{index rigidity}\cite{Wa}
Let $B\subset A$ be an  inclusion of unital $C^*$-algebras with a finite
index conditional expectation $E: A \rar B$.  If $\mathrm{Ind} (E)$ is
a scalar, then it takes values  in the set
\[
\left\{4{\cos}^2 \left(\frac{\pi}{n}\right),
n=3,4,5,\cdots\right\}\cup [4,\infty).
  \]
   In particular, if $B\subset A$ is an inclusion of simple
 unital $C^*$-algebras with $\mathrm{Ind} (E)<4,$ then there are no
 intermediate $C^*$-subalgebras of $B\subset A$.
\end{theorem}

 \begin{remark}\label{remarks}
   \begin{enumerate}
     \item The assumption that the inclusion has common identity is redundant,
  because, if $E: B \rar A$ is a conditional expectation of
  finite index, and $B$ is also a unital $C^*$-algebra with unit
  $1_B$, then
  \[
  1_A = \sum_i E(1_A \lambda_i)\lambda_i^* =
  \sum_i E(\lambda_i)\lambda_i^* = \sum_i 1_B E( \lambda_i)\lambda_i^*
= \sum_i E(1_B \lambda_i)\lambda_i^* = 1_B.
\]
\item A conditional expectation of finite index is automatically
  faithful. (It follows from \Cref{pipo-inequality}.)
\end{enumerate}
   \end{remark}

 We now recall a useful result which  says
 that  a conditional expectation with finite Watatani index also has
 finite probabilistic index (of Pimsner and Popa - see
 \cite{PiPo1}).
\begin{proposition}
 \label{probabilisticindex}\cite{Wa}\label{pipo-inequality}
 Let $B\subset A$ be an inclusion of $C^*$-algebras with a
 conditional expectation $E$ of finite index. Then, there exists
 a $c>0$ such that
 \begin{equation}\label{inequality-pipo}
 E(x)\geq cx ~~~\textrm{for all}~~~x\in A_{+}.
 \end{equation}
\end{proposition}
 Izumi showed that the converse also holds for inclusions of simple
unital $C^*$-inclusions.
\begin{theorem}\cite{I}\label{izumi}
 Let $B\subset A$ be an inclusion of simple unital $C^*$-algebras with
 a conditional expectation $E: A \rar B$. If $E$ satisfies the
 Pimsner-Popa inequality \eqref{inequality-pipo}, then $E$ has finite Watatani index.
  \end{theorem}
For  more on Watatani index, we suggest the reader to see \cite{Wa}.

\subsubsection{Minimal conditional expectations}\( \)

 Recall that if $B \subset A$ is an inclusion of unital $C^*$-algebras
 such that $\mathcal{Z}(A) = \C$, then every finite index conditional
 expectation has scalar index and a conditional expectation $E_0\in
 \mathcal{E}_0( A , B)$ is said to be minimal if it satisfies
 $\mathrm{Ind}(E_0) \leq \mathrm{Ind}(E)$ for all $E \in
 \mathcal{E}_0(A, B)$.  See \cite{H,Lon1,Lon2,Wa,KaWa} for details.

 Under some hypothesis, there exists only one conditional expectation.
\begin{theorem}\cite[Corollary 1.4.3]{Wa}\label{wata1}
Let $B\subset A$ be an inclusion of unital $C^*$-algebras and $E\in
\mathcal{E}_0(A,B)$. If $B^{\prime}\cap A \subseteq B$, then $E$ is
the unique conditional expectation from $A$ onto $B$.

 In particular, $E$ is a minimal conditional expectation from $A$
 onto $B$.
\end{theorem}
Interestingly, when the $C^*$-algebras are simple, then we have a unique
minimal conditional expectation.
\begin{theorem}\cite[Theorem 2.12.3]{Wa}\label{min1}
Let $B \subset A$ be an inclusion of simple unital $C^*$-algebras such that 
$\, \mathcal{E}_0(A,B) \neq \emptyset$. Then, there exists a
unique minimal conditional expectation from $A$ onto $ B$ (which will
be denoted by $E^A_B$).
\end{theorem}

\begin{definition}\cite{Wa}
Let $B \subset A$ be an inclusion of simple unital $C^*$-algebras such that
$\, \mathcal{E}_0(A,B) \neq \emptyset$. Then, its Watatani index is
defined as
\[
[A:B]_0:=\mathrm{Ind}(E^A_B).
\]
  \end{definition}

We now list two useful results related to composition of minimal
conditional expectations and multiplicativity of index.
\begin{theorem}\cite{KaWa}\label{min3}
Let $B\subset A$ be as in \Cref{min1}, $C$ be an intermediate simple
$C^*$-subalgebra of $B \subset A$, $F \in \mathcal{E}_0(A,C)$ and $E
\in \mathcal{E}_0(C,B)$. Then, $E\circ F$ is minimal if and only if
both $E$ and $F$ are minimal.

Moreover, the Watatani index is multiplicative, that is
${[A:B]}_0={[C:B]}_0{[A:C]}_0.$
\end{theorem}

\begin{lemma}\cite{Wa, I, IW}\label{intermediate-ce}
  Let $B \subset A$ and $E^A_B$ be as in \Cref{min1} and $C$ be an
  intermediate simple $C^*$-subalgebra of $B \subset A$. Then, there exist
  unique minimal conditional expectations $E^C_B: C \rar B$ and
  $E^A_C: A \rar C$, and they satisfy the relation $E^C_B\circ
  E^A_C=E^A_B$.

  Moreover, if $B \subset A$ is irreducible, i.e., $B'\cap A = \C$,
  then $E^C_B = {E^A_B}_{|_C}$.
\end{lemma} 
\begin{proof}
 Let $E:={E^A_B}_{|_C}$. Then, by \Cref{pipo-inequality}, $E$
 satisfies the Pimsner-Popa inequality. Hence, by \Cref{izumi}, $E \in
 \mathcal{E}_0(C,A)$. On the other hand, by \cite[Proposition
   6.1]{I}, there exists an $F \in \mathcal{E}_0(A,C)$.  Thus,
 $\mathcal{E}_0(C,A) \neq \emptyset \neq \mathcal{E}_0(A,C)$. So, by
 \Cref{min1}, there exist unique minimal conditional expectations
 $E^C_B: C \rar B$ and $E^A_C: A \rar C$. Then, by \Cref{min3}, $E^C_B
 \circ E^A_C: A \rar B$ is a minimal conditional expectation. Hence,
 by \Cref{min1} again, we must have $E^C_B \circ E^A_C = E^A_B$.

And, when $B \subset A$ is irreducible, then by \Cref{wata1}, we have
$E^C_B = {E^A_B}_{|_C}$.\end{proof}

\begin{example}\cite{Wa}
 Consider a unital simple $C^*$-algebra $B$ with a finite group $G$
 acting outerly on $B$ and consider the $C^*$-crossed product
 $ B\rtimes G$. Then, the canonical conditional expectation $E: B
 \rtimes G \rar B$ given by $E\big(\sum_g x_gu_g\big)=x_e$ is minimal,
 the proof of which can be read off \cite{Wa} and we omit the
 necessary details.
\end{example}

\subsection{Watatani's $C^*$-basic construction}\( \)

We now briefly recall the theory of \textit{$C^*$-basic construction}
introduced by Watatani in \cite{Wa}. Let $A$ be a  $C^*$-algebra
and $\mathcal{H}$ be a pre-Hilbert $A$-module. Recall that the map

\[
\mh \ni x \mapsto {\lVert
  x\rVert}_{\mh}:=\lVert \langle x,x\rangle_A \rVert ^{1/2} \in [0,
  \infty)
  \]
  is a norm on $\mh$; and that, $\mh$ is called a Hilbert $A$-module
  if it is complete with respect to this norm. For details about the theory of
  Hilbert $C^*$-modules, we refer the reader to \cite{lance}.

Now, suppose $B \subset A$ is a unital inclusion of 
$C^*$-algebras with a faithful conditional expectation $E_B$ from $A$
onto $B$. Then, $A$ becomes a pre-Hilbert $B$-module with respect to
the $B$-valued inner product given by
  \begin{equation}\label{B-valued}
  \langle x, y\rangle_{B}=E_B(x^*y) \text{ for all } x, y\in A.
  \end{equation}
 Here we follow the Physicists' convention of keeping conjugate
 linearity in the first coordinate.  Let $\mathfrak{A}$ denote the
 Hilbert $B$-module completion of $A$ and $\iota:A\rightarrow
 \mathfrak{A}$ denote the isometric inclusion map. Recall that the
 space $\mathcal{L}_{B}(\mathfrak{A})$ consisting of adjointable
 $B$-linear maps on $\mathfrak{A}$ is a $C^*$-algebra.

  For each $a \in A$, consider $\lambda(a)\in \mathcal{L}_B(\mathfrak{A})$
  given by $\lambda(a)\big(\iota(x)\big)=\iota(ax)$ for $x\in A$. The
  map $\iota (A) \ni \iota(x) \mapsto \iota (E_B(x))\in \iota (A)$
  extends to an adjointable projection on $\mathfrak{A}$,
  and is denoted by $e_B\in \mathcal{L}_B(\mathfrak{A})$. The
  projection $e_B$ is called the Jones projection for the pair $B
  \subset A$; thus, $e_B(\iota(x))=\iota(E_B(x))$ for all $x \in
  A$. The $C^*$-basic construction $C^*\langle A,e_B\rangle$ is
  defined to be the $C^*$-subalgebra generated by $\{\lambda(A),
  e_B\}$ in $\mathcal{L}_B(\mathfrak{A})$. It turns out that
  $C^*\langle A, e_B \rangle$ equals the closure of the linear span of
  $\{\lambda(x)e_B\lambda(y):x,y\in A\}$ in the $C^*$-algebra $
  \mathcal{L}_B(\mathfrak{A})$; $\lambda$ is an injective
  $*$-homomorphism and thus we can consider $A$ as a $C^*$-subalgebra of
  $C^*\la A, e_B\ra$. The following inequality, known as
  the Kadison-Schwarz inequality, holds:
  \begin{equation}\label{KS-inequality}
 E_B(x)^*E_B(x)\leq E_B(x^*x) \text{ for all }x\in A.
\end{equation}
Interestingly, when the conditional expectation has finite index then
$A$ turns out to be complete with respect to the above norm as we show
below.
\begin{lemma}\label{Aisahilbertmodule}\cite{Wa}
Let $B \subset A$ be a unital inclusion of  $C^*$-algebras and
$E_B \in \mathcal{E}_0(A,B)$.  Then, $A$ is a Hilbert $B$-module with
respect to the $B$-valued inner product given as in \Cref{B-valued}.
\end{lemma}

\begin{proof}
 Since a conditional expectation with finite index is faithful
 (\Cref{remarks}), $A$ is a pre-Hilbert $B$-module. By
 \Cref{pipo-inequality}, we have \( E_B(x^*x)\geq L\, x^*x \text{ for
   every } x\in A\) for some positive constant $L$. Therefore,
 ${\lVert x\rVert}_{A}\geq L \lVert x \rVert$ for all $x \in A$. In
 particular, if $\{x_n\}$ is a Cauchy sequence in $A$ with respect to
 ${\lVert.\rVert}_{A}$, then so is it with respect to $\lVert.\rVert$
 and, therefore, converges to some element $x\in A$. On the other
 hand, $\|y \|_A^2 = \|E_B(yy^*)\| \leq \|y\|^2 $ for all $y \in
 A$. So, $\{x_n\}$ converges to $x$ with respect to
 ${\lVert.\rVert}_{A}$ as well. Thus, $A$ is complete with respect to
 ${\lVert . \rVert}_{A}.$
\end{proof}

A simple algebraic calculation yields the following useful and
standard equality, and is left to the reader.
\begin{proposition}\label{support}\label{basis} Let $A, B$
  and $E_B$ be as in \Cref{Aisahilbertmodule} and $\{\lambda_i: 1 \leq i \leq n\}$ be a
  quasi-basis for $E_B$. Then,
\[ \sum_{i=1}^n \lambda_i e_B\lambda_i^* = 1. \]
\end{proposition}

\begin{theorem}\cite{Wa, KW} \label{dual-ce}\label{w}
  Let $A, B$ and $E_B$ be as in \Cref{Aisahilbertmodule} and let $A_1$
  denote the $C^*$-basic construction of $B \subset A$ with respect to
  $E_B $. Then, we have the following:
   \begin{enumerate}
\item 
  There exists a unique finite index conditional expectation
  $\widetilde{E}_B: A_1 \rar A$ satisfying
  \[
  \widetilde{E}_B\big(\lambda(x)e_B \lambda(y)\big) =
  \lambda(x)\lambda\big(\mathrm{Ind}(E_B)^{-1}\big) \lambda(y) =
  \lambda\big(\mathrm{Ind}(E_B)^{-1} x y \big) 
  \]
 for all $x, y \in A$. ($\widetilde{E}_B$ is called the dual conditional expectation of
 $E_B$.)\\  (\cite[Proposition 1.6.1]{Wa})

\item If $A$ and $B$ are both simple, then $A_1$ is also simple and if
  $E_0:A \rightarrow B$ denotes the unique minimal conditional
  expectation, then the dual conditional expectation $\widetilde{E}_0:
  A_1 \rar A$ is minimal as well and $\mathrm{Ind} (E_0) =
  \mathrm{Ind} (\widetilde{E}_0)$.\hspace*{10mm}  (\cite[2.2.14 and 2.3.4]{Wa},
  \cite{KW})\\
  \end{enumerate}
  \end{theorem}

 \begin{remark}\label{dual-index}
 By \Cref{w}, we have $[A_1:A]_0 = [A:B]_0$.
  \end{remark}
The following lemma is an extremely useful observation and is a direct
adaptation of the so called ``Push-down Lemma'' from \cite{PiPo1}.
\begin{lemma}\label{pushdown}   Let $A, B$, $E_B$, $A_1$ and $\widetilde{E}_B$ be as in \Cref{w}.
 If $x_1\in A_1$, then there exists a unique  $x_0\in A$ such
 that $x_1e_1=x_0e_1$; this element is given by
 $[A:B]_0 \widetilde{E}_B(x_1e_1)$.
\end{lemma}
\begin{proof}
 Uniqueness is trivial. Suppose $\{\lambda_i:1\leq i\leq n\}$ is a
 quasi-basis for $E_B$. Then, by \cite[ Proposition 1.6.6]{Wa},
 $\{{{[A:B]}_0}^{1/2}\lambda_ie_1: 1 \leq i \leq n\}$ is a quasi-basis for
 $\widetilde{E}_B$. Therefore, $A_1=Ae_1A$ $:=\text{span}\{ae_1b: a, b
 \in A\}$; and, it is easy to see that
 ${{[A:B]}_0}\widetilde{E}_B(ae_1be_1)e_1=ae_1be_1$ for all $a, b \in
 A$. This completes the proof.
\end{proof}

\subsection{Iterated $C^*$-basic constructions and the relative commutants}\( \)

Throughout this subsection, $B \subset A$ will
denote a fixed pair of simple unital $C^*$-algebras such that
$\mathcal{E}_0(A,B)\neq \emptyset$; and $\tau:={{[A:B]}_0}^{-1}.$

From \Cref{w}, we know that $A \subset A_1$ is also a pair of simple
unital $C^*$-algebras (with common identity) and that $[A_1: A]_0 =
[A:B]_0$. Thus, like Jones' tower of basic constructions of a finite
index subfactor of type $II_1$, we can repeat the process of
$C^*$-basic construction to obtain a tower of simple unital
$C^*$-algebras
\begin{equation}\label{b-c-tower}
B \subset A \subset A_1 \subset A_2 \subset \cdots \subset A_k \subset \cdots
\end{equation}
with unique (dual) minimal conditional expectations $E_{k}: A_{k} \rar
A_{k-1}$, $k \geq 0$, with the convention that $A_{-1}:=B$ and
$A_0:=A$. We shall call this tower {\it the tower of $C^*$-basic
  constructions} of the inclusion $B \subset A$. For each $k \geq 1$,
let $e_k$ denote the Jones projection in $A_k$ which implements the
$C^*$-basic construction of the inclusion $A_{k-2}\subset A_{k-1}$
with respect to the (minimal) conditional expectation $E_{k-1}:
A_{k-1} \rar A_{k - 2}$. For each $k \geq 0$, the relative commutants
of $B$ in $A_k$ is given by
\begin{equation}
B'\cap A_k=\{ x \in A_k: x b = b x \text{ for all } b \in B\}.
\end{equation}

  \begin{proposition}\cite{Wa}\label{rel-comm}
    $B'\cap A_k$ is finite dimensional for all $k \geq 0$.
  \end{proposition}
  \begin{proof}
  Since $A_k$ is simple and the conditional expectation
$E_{0} \circ E_1 \circ \cdots \circ E_{k-1} \circ E_k : A_k \rar B$
has finite index (by \cite[Proposition 1.7.1]{Wa}), it follows from
\cite[Proposition 2.7.3]{Wa} that $B'\cap A_k$ is finite dimensional.
\end{proof}
We shall provide a bound for the dimension of $B'\cap A_k$ in terms of
index of $B \subset A$ in the next subsection.
 \subsubsection*{$k$-step $C^*$-basic construction}\( \)

The multi-step basic construction holds exactly like in
\cite{PiPo1}. See \cite{Bak} for an easier proof. Out here, we use the
characterization of $C^*$-basic construction given by Watatani in
\cite[Proposition 2.2.11]{Wa}.

\begin{proposition}\cite{PiPo1}\label{iteratingbasicconstruction}
For each $n\geq 1$, the tower $B\subset A_n \subset A_{2n+1}$ is an
instance of $C^*$-basic construction with the corresponding Jones
projection  given by
\[
e_{[-1,2n+1]}:={{\tau}^{-\frac{n(n+1)}{2}}}(e_{n+1}e_n\cdots e_1)
(e_{n+2}e_{n+1}\cdots e_2)\cdots(e_{2n+1}e_{2n}\cdots e_{n+1}).
\]
\end{proposition}

\begin{proposition}\cite{JS, Wa}\label{multibasis}
 Let $\{\lambda_i:1 \leq i\leq m \}$ be a quasi-basis for $E_0$. Then, for each
 $n \geq
 1$, the collection
 \[
 \bigg\{{\tau}^{-\frac{n(n+1)}{4}}\lambda_{i_{n}}(e_1e_2\cdots
 e_{n-1}e_n)\lambda_{i_{n-1}}(e_1e_2\cdots e_{n-2}e_{n-1})\cdots
 \lambda_{i_2}(e_1e_2)\lambda_{i_1}e_1: 1\leq i_1, \ldots, i_n\leq m\bigg\}
 \]
 is a quasi-basis for the minimal conditional expectation $E_0\circ
 E_1\circ \cdots E_{n-1}\circ E_n: A_n \rar B$.
 \end{proposition}

\subsection{Tracial states on the relative commutants}\( \)

Like in the preceding subsection, $B \subset A$ will again denote a fixed pair
of simple unital $C^*$-algebras such that $\mathcal{E}_0(A,B)\neq
\emptyset$; and $\tau:={{[A:B]}_0}^{-1}.$

Being finite dimensional, the higher relative commutants $B'\cap A_k$
admit numerous tracial states. However, using the minimal conditional
expectations, we can single out a consistent Markov type trace, which
then allows one to talk about the ``standard invariant'' and the
``principal graph'' of such an inclusion, as is done for any finite
index subfactor. Izumi has also mentioned about this aspect in
\cite{I}. But we are not aware of any literature in this
direction. However, we will not delve into these topics in this
paper.\smallskip

First, we recall two auxiliary results from \cite{Wa} that will be
fundamental in obtaining the tracial states of our choice.
\begin{proposition}\cite{Wa}\label{min4}
 Let $E\in \mathcal{E}_0(A, B)$ and $\{\lambda_i:1\leq i\leq n \}$ be
 a quasi-basis for $E$. Consider the map $H_E: B'\cap A \rar A$ given
 by
\[
H_E(x)=\sum_{i}
\lambda_ix\lambda^*_i,\ x\in B^{\prime}\cap
A.
\]
Then, $H_E$ maps $B^{\prime}\cap A$ onto $\mathcal{Z}(A)$ and does not
depend on the choice of the quasi-basis.

Moreover, the map $G_E: B^{\prime}\cap A \rar \mathcal{Z}(A)$  given by 
\[
G_E(x)=\frac{1}{\mathrm{Ind}(E)}\sum_{i}
\lambda_ix\lambda^*_i,\ x \in B'\cap A
\]
is a conditional expectation.
\color{black}
\end{proposition}

\begin{theorem}\cite{Wa}\label{min5}
Let $E\in \mathcal{E}_0(A,B)$. Then, the
following are equivalent:
\begin{enumerate}
 \item $E$ is minimal.
 \item $E_{|_{B^{\prime}\cap A}}$ (resp., $H_E$) is a tracial state
   (resp., tracial map) on $B^{\prime}\cap A$ and
 $$H_E= \mathrm{Ind}(E)\ E|_{B^{\prime}\cap A}.$$
 \item $H_E=c \, E|_{B^{\prime}\cap A}$ for some constant $c$.
\end{enumerate}
\end{theorem}

 \begin{proposition}\label{f1}
 For each $k \geq 0$, $B^{\prime}\cap A_k$ admits a faithful tracial
 state $\mathrm{tr}_k$ such that
  \begin{equation}\label{tr-xe}
    \mathrm{tr}_k(xe_k)=\tau
    \mathrm{tr}_k(x) \text{ for all } x\in B^{\prime}\cap A_{k-1},
    \end{equation}
and ${\tr_k}_{|_{B'\cap A_{k-1}}} = \tr_{k-1}$ for all $ k \geq
1$. (We will drop $k$ and  denote $\tr_k$ simply  by $\tr$.)
 \end{proposition}
 \begin{proof}
Define $\tr_k: B'\cap A_k\rar \C$ as $\tr_k = (E_0\circ E_1\circ\cdots
\circ E_k)_{|_{B^{\prime}\cap A_k}}$. Then, by \Cref{min3},
\Cref{min5} and \Cref{remarks}(2), $\tr_k$ is a faithful tracial state
and, by definition, ${\tr_k}_{|_{B'\cap A_{k-1}}} = \tr_{k-1}$ for all
$ k \geq 1$.

   We prove the Markov type property only for $k=1$. Other cases follow
   similarly. We have
 \[
 \mathrm{tr}(xe_1)=E_0\circ E_1(xe_1)= E_0\big(xE_1(e_1)\big)=\tau
 E_0(x)=\tau \mathrm{tr}(x)
 \]
 for all $x \in B'\cap A$.
 \end{proof}

\begin{remark}
Denote the minimal conditional expectation $E_0\circ E_1\circ\cdots
\circ E_k$ simply by $F_k$. Then, in view of Theorem \ref{min5},
$H_{F_k}={\tau}^{-k} \mathrm{tr}_k$ on $B'\cap A_k$.
\end{remark}

\begin{lemma}\label{f2}
Let $\{\lambda_i:1\leq i\leq n\}\subset A$ be a quasi-basis for the
minimal conditional expectation $E_0$. Then, the
$\mathrm{tr}$-preserving conditional expectation from $B^{\prime}\cap
A_k$ onto $A^{\prime}\cap A_k$ is given by
\[
E^{B^{\prime}\cap A_k}_{A^{\prime}\cap A_k}(x)=
\frac{1}{[A:B]_0}\sum_{i}\lambda_ix\lambda^*_i,\  x \in B'\cap A_k.
\]
 \end{lemma}
\begin{proof}
   Consider $G_{E_0}: B'\cap A_k \rar A_k$ given by
  \[
G_{E_0}(x)=\tau \sum_{i}\lambda_i x \lambda^*_i, x \in B'\cap A_k.
  \]
We assert that $G_{E_0}(x)\in A^{\prime}\cap A_k$. Indeed, for any
$a\in A$ and $x\in B'\cap A_k$, as in \cite[ Proposition 1.2.9 ]{Wa},
we observe that
\begin{align*}
  G_{E_0}(x)a & = \tau\sum_{i}\lambda_ix\lambda^*_ia\\ &= \tau
  \sum_{i}\lambda_ix\bigg(\sum_{j\in
    J}E_0(\lambda^*_ia\lambda_j)\lambda^*_j\bigg)\\ &= \tau \sum_{i,j}
  \lambda_iE_0(\lambda^*_ia\lambda_j)x\lambda^*_j\\ &=a\sum_{j}
  \lambda_jx\lambda^*_j\\ &= a\, G_{E_0}(x).
 \end{align*}
  Now, we show that $G_{E_0}$ is the $\mathrm{tr}$-preserving
  conditional expectation $E^{B^{\prime}\cap A_k}_{A^{\prime}\cap
    A_k}$. Let $x \in B'\cap A_k$. Then, by the definition of $\tr$, for
  any $y\in A^{\prime}\cap A_k$, we have
  \[
  \mathrm{tr}\left(\sum_i\lambda_ix\lambda^*_iy\right)=F_k\left(\sum_i
 \lambda_ix\lambda^*_iy\right)=F_k\left(\sum_i \lambda_ixy\lambda^*_i\right).
 \]
 Clearly, $E_1\circ \cdots \circ E_k(xy)\in B^{\prime}\cap A$. 
 Hence, in view of Theorem \ref{min5}, we observe that  $H_{E_0}\big(E_1\circ
 \cdots \circ E_k(xy)\big)\in A^{\prime}\cap A=\C$ .
 Thus,
 \[
 F_k\left(\sum_i
 \lambda_ixy\lambda^*_i\right)=E_0\bigg(\sum_i\lambda_i\big(E_1\circ \cdots
 \circ E_k(xy)\big)\lambda^*_i\bigg)=H_{E_0}\bigg(E_1\circ \cdots
 \circ E_k(xy)\bigg).
 \]
 Therefore,  $\mathrm{tr}(\sum_i\lambda_ix\lambda^*_iy)=
       {\tau}^{-1}F_k(xy)={\tau}^{-1}\mathrm{tr}(xy)$,  again by Theorem
       \ref{min5}.  This proves that $G_{E_0}= E^{B^{\prime}\cap
         A_k}_{A^{\prime}\cap A_k}.$
\end{proof}
In view of \Cref{support}, we have:
\begin{corollary}\label{E-e}
 \(
  E^{B^{\prime}\cap A_1}_{A^{\prime}\cap
    A_1}(e_1)= \tau.
  \)
    \end{corollary}

We now provide a bound for the dimension of each relative commutant
$B'\cap A_k$, whose proof is motivated by that of \cite[Lemma
  3.6.2(b)]{GHJ} (see \cite{Lon1} for the type $III$ case). We will need the following observation related to the 
local behaviour of conditional expectation and index.

\begin{proposition}\label{local}
  For each non-zero projection $p$ in $B'\cap A$, consider the $C^*$-inclusion
  $pBp \subset pAp$ and the map $E_p: pAp \rar pBp$ given by
\[
E_p(x) = \frac{E_0(x)p}{\tr (p)},\ x \in pAp.
\]
Then, the following hold:
  \begin{enumerate}
\item The pair $pBp\subset pAp$ is an inclusion of simple unital
  $C^*$-algebras with common identity $p$.
    \item $E_p$ is a conditional expectation of finite index and, for
      any quasi-basis $\{\lambda_i: 1 \leq i \leq n\}$ of $E_0$,
      $\{\sqrt{\tr(p)}~ p\lambda_ip: 1 \leq i \leq n\}$ is a
      quasi-basis for $E_p$.
\item $E_p$ is the unique minimal conditional expectation from $pAp$
    onto $pBp$.
\item  ${[pAp:pBp]}_0={\tr(p)}^2 {[A:B]}_0\, p$.
\item Suppose $\{p_i\}_{i=1}^n$ is a partition of identity consisting
  of projections in $B^{\prime}\cap A.$ Then,
  ${{[A:B]}_0}^{1/2}=\sum_{i=1}^n {\big \lVert
    {[p_iAp_i:Bp_i]}_0\big\rVert}^{1/2}.$
\end{enumerate}
  \end{proposition}
\begin{proof}
(1): It is well known that $pAp$ is a hereditary $C^*$-subalgebra of
  $A$.  Since any hereditary $C^*$-subalgebra of a simple
  $C^*$-algebra is simple, $pAp$ is simple. On the other hand, since
  $p \in B'\cap A$, $B \cong Bp$; so, $pBp$ is also simple.  Hence,
  $pBp \subset pAp$ is an inclusion of simple $C^*$-algebras with
  common identity $p$.
\smallskip

(2): Clearly, $E_p$ is a conditional expectation. And, for any $x\in pAp$,
we observe that
 \begin{align*}
   \sum_i E_p(xp\lambda_ip)p\lambda^*_ip &=\frac{1}{
     \tr(p)}\sum_i
  E_0(xp\lambda_ip)p\lambda^*_ip\\ &= \frac{1}{\tr (p)}\sum_i
  E_0(x\lambda_ip)p\lambda^*_ip\\ &= \frac{1}{\tr (p)}\sum_i
  E_0(px\lambda_i)p\lambda^*_ip \qquad \qquad \text{ (by \cite[Lemma
      3.11]{KW})}\\ &= \frac{1}{\tr (p)}\,  p
  \sum_i E_0(x\lambda_i)\lambda^*_ip \\
  & = \frac{1}{\tr (p)}\, x.
 \end{align*}
Hence, $E_p$ has finite index with a quasi-basis $\{\sqrt{\tr (p)}~
p\lambda_ip: 1 \leq i \leq n\}$.\smallskip

  (3): For each $x\in (pBp)^{\prime}\cap pAp$, we have
 \begin{align*}
  H_{E_p(x)} 
  &= \sum_i \tr(p)\,  p\lambda_ip x p\lambda^*_ip\\
  &= \tr(p)\,  p\big(\sum_i \lambda_i pxp\lambda^*_i\big)p\\
  &= \tr(p)\,  pH_E(x)p\\
  &= c\, E_p(x). \qquad\qquad \qquad (\text{for some constant } c)
 \end{align*}
Therefore, from \Cref{min5}, we conclude that $E_p$ is a minimal
conditional expectation. Uniqueness follows from \Cref{min1}. \smallskip

(4): We have 
 \begin{align*}
 {[pAp:pBp]}_0 &= \mathrm{Ind}(E_p)\\ & = \tr(p) \sum_i p\lambda_ip\lambda^*_ip\\ &= \tr(p)\, 
 p \big(\sum_i \lambda_ip\lambda^*_i\big)p\\ &= \tr(p)\,  pH_{E_0}(p) p\\ &=
 \tr(p)^2\, {[A:B]}_0\, p, \,
 \end{align*}
 where the last equality follows from the fact that $H_{E_0}(p)=\mathrm{Ind}(E_0)\, E_0(p)=
\mathrm{Ind}(E_0)\, \tr(p),$ by  \Cref{min5}. \smallskip

(5) follows   readily  from Item (4).
\end{proof}

\begin{proposition}\label{rel-comm-dimn}
We have
 \begin{equation}\label{rel-dim-k}
    \mathrm{dim}_{\C} (B^{\prime}\cap A_{k})\leq  {[A:B]_0}^{k+1}
  \end{equation}
  for all $k\geq 0$. 
 \end{proposition}

\begin{proof}
  Since Watatani index is multiplicative (\Cref{dual-index}), it
  suffices to prove \eqref{rel-dim-k} for $k = 0$.

 Let $\{p_i:1\leq i\leq m\}$ be a maximal family of mutually orthogonal
minimal projections in $B^{\prime}\cap A$ such that $\sum_{i=1}^m p_i=1$.

Note that, for each projection $p$ in $B'\cap A$, by \Cref{local} and
\cite[Lemma 2.3.1]{Wa}, we have
\[
 {[A:B]}_0\, p = \frac{{[pAp:pBp]}_0}{{\tr(p)}^2}\geq
 \frac{p}{{\tr(p)}^2};
 \]
 so that $\tr(p) {[A:B]}_0\geq \frac{1}{\tr(p)}$.  Thus,
 ${[A:B]}_0\geq \sum_{i=1}^m \frac{1}{tr(p_i)}.$ Since $\sum_{i=1}^m
 \tr(p_i)=1$, it follows that $\sum_{i=1}^m \frac{1}{\tr(p_i)}\geq
 m^2$. Hence,
\[
\text{dim}_{\C} (B^{\prime}\cap
A) \leq m^2\leq \sum_{i=1}^m\frac{1}{\tr(p_i)}\leq {[A:B]}_0.
\]
This completes the proof.\smallskip
\end{proof}

\begin{corollary}
 If ${[A:B]}_0<4$, then $B\subset A$ is irreducible.
\end{corollary}
\begin{proof}
 From the proof of \Cref{rel-comm-dimn} it follows that if $\{p_i:1\leq i\leq m\}$ is a maximal family of mutually orthogonal
minimal projections in $B^{\prime}\cap A$ such that $\sum_{i=1}^m p_i=1$, then 
$m^2\leq {[A:B]}_0.$ Thus, if ${[A:B]}_0< 4$, then  we must  have $m=1.$ This completes the proof.
\end{proof}

\section{Fourier theory for inclusions of
  simple unital $C^*$-algebras}

In the theory of subfactors, some of the crucial tools include certain
naturally occurring operations on the higher relative commutants,
namely, the so-called {\it Fourier transforms, shift operators} and
{\it rotation maps}. These were introduced by Ocneanu (see \cite{Oc})
and played a major role in the development of the subject. Details may
be found in \cite{Bi,Bi2,BJ}. A significant application of the Fourier
theory has been that the {\it rotation} maps on the higher relative
commutants were highly instrumental in the formalism of the structure
of Jones' {\it planar algebra} {on the} {\it standard invariant} of
any {\it extremal} subfactor (see \cite{Jo}). According to Jones (see
\cite{Jo2}), {the rotation operator is `the most interesting algebraic
  ingredient of a subfactor seen from the planar point of view'.} The
formulation of {\it Fourier transforms} and {\it rotation maps} for a
subfactor $N \subset M$ depends heavily on the unique tracial state on
the $II_1$ factor $M$ and the modular conjugation operator $J$ on
$L^2(M)$. Needless to mention, both of these tools are absent for
general inclusions of simple unital $C^*$-algebras. Still, given any
pair $B \subset A$ of simple unital $C^*$-algebras, based on the
consistent tracial states on the relative commutants that we obtained
in the preceding section, we will show that an analogous Fourier
theory can be developed.

As an application, we shall provide bounds for the cardinality of the
lattice of intermediate subalgebras of such pairs of $C^*$-algebras as
well as of subfactors of type $III$. We believe that, very much like
$II_1$-factors and their subfactors, the $C^*$-Fourier theory will
also have a significant say in the understanding of simple unital
$C^*$-algebras and their $C^*$-subalgebras.  \color{black}
\bigskip

As in the preceding section, throughout this section too, $B \subset
A$ will denote a fixed pair of simple unital $C^*$-algebras such that
$\mathcal{E}_0(A, B) \neq \emptyset$; and $\tau:=[A:B]_0^{-1}$.

\subsection{Fourier transforms}\( \) 

\begin{definition}
For each $k\geq 0$, the Fourier transform $\mathcal{F}_k:
B^{\prime}\cap A_k \rar A^{\prime}\cap A_{k+1}$ is defined as 
\[
\mathcal{F}_k(x)={\tau}^{-\frac{k+2}{2}}E^{B^{\prime}\cap
  A_{k+1}}_{A^{\prime}\cap A_{k+1}}(xe_{k+1}e_k\cdots e_2 e_1), \, x \in B'\cap A_k.
\]
And,   the inverse Fourier transform 
 ${\mathcal{F}}^{-1}_k: A^{\prime}\cap A_{k+1} \rightarrow
 B^{\prime}\cap A_k$ is defined as 
 \[
   {\mathcal{F}}^{-1}_k(x)= {\tau}^{-\frac{k+2}{2}}E_{k+1}(xe_1e_2\cdots e_ke_{k+1}),\, x \in A'\cap A_{k+1}.
   \]
 \end{definition}

\noindent The usage of the word ``inverse'' in the
  preceding definition is justified by the following:
  \begin{proposition} \label{inverse}
    We have
    \[
 \mathcal{F}_k \circ {\mathcal{F}}^{-1}_k= \mathrm{Id}_{A'\cap
   A_{k+1}} \text{ and } {\mathcal{F}}^{-1}_k\circ \mathcal{F}_k =
 \mathrm{Id}_{B'\cap A_k}
 \]
 for all $ k \geq 0$. In particular, if $B\subset A$ is irreducible,
 then so is $A \subset A_1$.
\end{proposition}
\begin{proof}
 First,  observe that,
for any $a\in A$, we have
 \begin{equation}\label{one}
  (e_{k+1}e_k\cdots e_2e_1) a (e_1e_2\cdots e_ke_{k+1})={\tau}^kE_0(a)e_{k+1}.
 \end{equation}
 Similarly, it is easy to see that 
 \begin{equation}\label{three}
 (e_1e_2\cdots e_ke_{k+1})(e_{k+1}e_k\cdots e_2e_1)={\tau}^k e_1.
  \end{equation}
Also, notice that if $\{\lambda_i:1\leq i\leq n\}$ is a
quasi-basis for $E_0$, then using Lemma \ref{f2}, we readily
obtain
 \begin{equation}\label{two}
   E^{B^{\prime}\cap A_{k+1}}_{A^{\prime}\cap
     A_{k+1}}(xe_{k+1}e_k\cdots e_2e_1)=\tau \sum_i \lambda_i
   xe_{k+1}e_k\cdots e_2e_1\lambda^*_i
    \end{equation}
 for all $x \in B'\cap A_{k+1}
   $. Then, for any $x\in B^{\prime}\cap A_k$,  we have
\begin{align*}
 {\mathcal{F}}^{-1}_k\mathcal{F}_k(x) &= {\tau}^{-(k+2)}
 E_{k+1}\big(E^{B^{\prime}\cap A_{k+1}}_{A^{\prime}\cap
   A_{k+1}}(xe_{k+1}e_k\cdots e_2e_1)e_1e_2\cdots e_ke_{k+1}\big)\\ &=
 {\tau}^{-(k+1)} E_{k+1}\big(\sum_i \lambda_i xe_{k+1}e_k\cdots
 e_2e_1\lambda^*_ie_1e_2\cdots e_ke_{k+1}\big) \hspace*{15mm} \textrm{(by
     Eq.}~~\eqref{two})\\ &= {\tau}^{-(k+1)}{\tau}^k \sum_i \lambda_ix
 E_{k+1}\big(E_0(\lambda^*_i)e_{k+1}\big)  \hspace*{30mm} \textrm{(by
     Eq.}~~\ref{one})\\ &= {\tau}^{-(k+1)}{\tau}^k\tau \sum_i
 \lambda_i E_0(\lambda^*_i)x   \hspace*{20mm} \textrm{(since}~~x\in
   B^{\prime}~~~~\textrm{and}~~E_{k+1}(e_{k+1})=\tau)\\ &=
 x.
\end{align*}
On the other hand, for any $y\in A^{\prime}\cap A_{k+1}$, we see that
\begin{align*}
 \mathcal{F}_k {\mathcal{F}}^{-1}_k(y) &= {\tau}^{-(k+2)}
 E^{B^{\prime}\cap A_{k+1}}_{A^{\prime}\cap
   A_{k+1}}\big(E_{k+1}(ye_1e_2\cdots e_ke_{k+1})e_{k+1}e_k\cdots
 e_2e_1\big)\\ &= {\tau}^{-(k+1)} E^{B^{\prime}\cap
   A_{k+1}}_{A^{\prime}\cap A_{k+1}}\big(y(e_1e_2\cdots
 e_ke_{k+1})(e_{k+1}e_k\cdots e_2e_1)\big) \hspace*{10mm} \textrm{(by
   Lemma}~~\ref{pushdown})\\ &= {\tau}^{-1} E^{B^{\prime}\cap
   A_{k+1}}_{A^{\prime}\cap A_{k+1}}(ye_1) \hspace*{65mm}\textrm{(by
   Eq.}~~\eqref{three})\\ &= y.
 \end{align*}
This completes the proof.
\end{proof}

We now proceed to show that the Fourier transform $\mathcal{F}_1$ and
its inverse are both isometries. First, a lemma that will be required.

\begin{lemma}
 \label{f3}\label{f4}
Let $\{\lambda_i:1\leq i\leq n\}$ be a
quasi-basis for $E_0$. Then, for any two elements $x,y\in B^{\prime}\cap A_1$, we have
 \begin{enumerate}
   \item $\sum_iE_0(\lambda_i)E_1(y^*\lambda^*_ix)=E_1(y^*x)$; and
   \item $\sum_i \lambda_ie_1E_1(y^*\lambda^*_ix)$ is independent of
     the quasi-basis for $E_0$ and belongs to $ B^{\prime}\cap A_1.$
     \end{enumerate}
 \end{lemma}
\begin{proof} (1): We have
 \begin{align*}
  \sum_iE_0(\lambda_i)E_1(y^*\lambda^*_ix) &=
  \sum_iE_1\big(E_0(\lambda_i)y^*\lambda^*_ix\big)\\ &= \sum_i
  E_1\big(y^*E_0(\lambda_i)\lambda^*_ix\big)~~~~~\hspace{25mm}
  (\textrm{since}~~y\in B^{\prime})\\ &=
  E_1\Big(y^*\big(\sum_iE_0(\lambda_i)\lambda^*_i\big)x\Big)\\ &=E_1(y^*x).
 \end{align*}

(2): First, we show that the operator $t:= \sum_i
 \lambda_ie_1E_1(y^*\lambda^*_ix)$ is independent of the quasi-basis
 for $E_0$. Suppose $\{\mu_j:1 \leq j\leq m\}$ is some other quasi-basis for
 $E_0$. Then,
 \begin{align*}
 \sum_i \lambda_ie_1E_1(y^*\lambda^*_ix) &= \sum_{j,i}
 \mu_jE_0(\mu^*_j\lambda_i)e_1E_1(y^*\lambda^*_ix)\\ &= \sum_{j,i}
 \mu_je_1E_0(\mu^*_j\lambda_i)E_1(y^*\lambda^*_ix)\\ &= \sum_{j,i}
 \mu_je_1E_1\big(E_0(\mu^*_j\lambda_i)y^*\lambda^*_ix\big)\\ &=
 \sum_{j,i}
 \mu_je_1E_1\big(y^*E_0(\mu^*_j\lambda_i)\lambda^*_ix\big)~~~~~\hspace{25mm}
    (\textrm{since}~~y\in B^{\prime})\\ &= \sum_{j}
    \mu_je_1E_1\Big(y^*\big(\sum_i
    E_0(\mu^*_j\lambda_i)\lambda^*_i\big)x\Big)\\ &=
    \sum_j\mu_je_1E_1(y^*\mu^*_jx).
 \end{align*}
Now, fix an $u\in \mathcal{U}(B)$. Then, clearly,  $\{u\lambda_i\}$ is also a quasi-basis for $E_0$. Therefore, 
\[
t=\sum_i u \lambda_ie_1E_1(y^*\lambda^*_iu^*x)=u \sum_i \lambda_ie_1E_1(y^*\lambda^*_ix)u^*=utu^*.
\]
Since $u$ was fixed arbitrarily,  $t\in B^{\prime}\cap A_1.$ This completes the proof.
\end{proof}

\begin{notation}
For simplicity,  we denote the Fourier transform $\mathcal{F}_1$ by $\mathcal{F}.$
\end{notation}

\begin{theorem}\label{F}
 $\mathcal{F}$ and
${\mathcal{F}}^{-1}$ are isometries with respect to the norm given by $\|x\|_2 = \mathrm{tr}(x^*x)$.
\end{theorem}
\begin{proof}
 Since $\mathcal{F}$ is a linear isomorphism, it suffices to show that
 one of them is an isometry.  Let $\{\lambda_i:1\leq i\leq n\}$ be a
 quasi-basis for $E_0$ and $x,y\in B^{\prime}\cap A_1$. Then,
\begin{align*}
 \big \langle \mathcal{F}(x),\mathcal{F}(y)\big \rangle &=
 \mathrm{tr}\Big(\big(\mathcal{F}(x)\big)^*\mathcal{F}(y)\Big)\\
 &=       {\tau}^{-3} \mathrm{tr}\Big(E^{B^{\prime}\cap
          A_2}_{A^{\prime}\cap A_2}(e_1e_2x^*)E^{B^{\prime}\cap
          A_2}_{A^{\prime}\cap A_2}(ye_2e_1)\Big)\\
        &=
        {\tau}^{-3}\mathrm{tr}\Big(E^{B^{\prime}\cap
          A_2}_{A^{\prime}\cap A_2}\big(E^{B^{\prime}\cap
          A_2}_{A^{\prime}\cap A_2}(e_1e_2x^*)ye_2e_1\big)\Big)\\
        &=
        {\tau}^{-3} \mathrm{tr}\Big(E^{B^{\prime}\cap
          A_2}_{A^{\prime}\cap
          A_2}(e_1e_2x^*)ye_2e_1\Big)\\
        &={\tau}^{-2} \mathrm{tr}
        \Big(\sum_i \lambda_ie_1e_2x^*\lambda^*_iye_2e_1\Big)
        ~~\hspace{27mm} (\textrm{by Lemma}~~\ref{f2})\\
        &={\tau}^{-2}
        \mathrm{tr}\Big(e_1\big(\sum_i
        \lambda_ie_1E_1(x^*\lambda^*_iy)\big)e_2\Big)\\
        &={\tau}^{-1}
        \mathrm{tr}\Big(e_1\sum_i
        \lambda_ie_1E_1(x^*\lambda^*_iy)\Big) ~~\hspace{20mm}
        (\textrm{by Lemmas}~~\ref{f4}~~\textrm{and}~~\ref{f1})\\ &=
               {\tau}^{-1} \mathrm{tr} \Big(e_1\sum_i
               E_0(\lambda_i)E_1(x^*\lambda^*_iy)\Big)\\ &=
               \mathrm{tr} \big(E_1(x^*y)\big)~~~~~~~\hspace{55mm}
               (\textrm{by \Cref{f3}})\\
                 &=\mathrm{tr}(x^*y)\\
                 &=\langle x,
                      y\rangle.
\end{align*}
This completes the proof.
\end{proof}
\begin{remark}
An astute reader must have noted that the analogues of \Cref{f3} and \Cref{F} are in fact
true for all $k \geq 0$. Some amount of book keeping will do the job. Since we need it only for $k=1$, we
leave the details to the reader.
\end{remark}
\subsection{Rotation maps and shift operators}\( \)

 Following \cite{Bi}, we show that the relative
 commutants can be endowed with certain \textit{rotation} maps.
Throughout this subsection, $\{\lambda_i: 1 \leq i \leq n\}$ will
  denote a fixed quasi-basis for the minimal conditional expectation $E_0: A
  \rar B$.

  \begin{definition}
    For each $k \geq 0$, the rotation map ${\rho}^{B\subset
      A}_k:B^{\prime}\cap A_k\rightarrow B^{\prime}\cap A_k$ is defined as
    \begin{equation}
      \rho^{B\subset A}_k(x)={\Bigg({\mathcal{F}}^{-1}_k\big({\mathcal{F}_k(x)
        }^*\big)\Bigg)}^*,\ x \in B'\cap A_k.
      \end{equation}
\end{definition}

 \begin{remark}\label{formulaforrotations}
By Lemma \ref{f2}, it is easily  seen that
\begin{equation}
\rho^{B\subset A}_k(x)={\tau}^{-k} \sum_i E_k(e_ke_{k-1}\cdots e_2e_1\lambda_ix)e_ke_{k-1}\cdots e_2e_1\lambda^*_i
\end{equation}
for all  $x \in B'\cap A_k$. 
 \end{remark}
For simplicity, we will focus mainly on $k=0$ and $k=1$ only. The higher
cases are straightforward generalizations and left to the reader.
First, we show that a certain square root exists for $\rho^{B\subset
  A_1}_1$.

\begin{proposition}\label{squareroot}
${\big(\rho^{B\subset A}_{3}\big)}^{2}= \rho^{B\subset A_{1}}_1.$
\end{proposition}
\begin{proof}
We have $\rho^{B\subset A}_{3}(x)={\tau}^{-3}\sum_i
E_3(e_3e_2e_1\lambda_ix)e_3e_2e_1\lambda^*_i$ for all $x \in B'\cap
A_3$.  On the other hand, by \Cref{iteratingbasicconstruction}, we
know that $B \subset A_1 \subset A_3$ is an instance of basic
construction.  So, from \Cref{formulaforrotations} and
\Cref{multibasis}, we have
$$\rho^{B\subset A_1}_1(x)={\tau}^{-5} \sum_{i,j} E_2\circ
E_3(e_2e_1e_3e_2\lambda_ie_1\lambda_jx)e_2e_1e_3e_2\lambda^*_je_1\lambda^*_i$$
for all $x \in B'\cap A_3$. Thus, for any $x\in B^{\prime}\cap A_3$, we obtain
\begin{align*}
 \big(\rho^{B\subset A}_3(x)\big)^2 
 &= \left({\tau}^{-3}\sum_i E_3\big(e_3e_2e_1\lambda_i \rho^{B\subset A}_3(x)\big)e_3e_2e_1\lambda^*_i\right)^2\\
 &= {\tau}^{-6} \sum_{i,j} E_3\bigg(e_3e_2e_1\lambda_iE_3(e_3e_2e_1\lambda_jx)e_3e_2e_1\lambda^*_j\bigg)e_3e_2e_1\lambda^*_i\\
 &= {\tau}^{-6} \sum_{i,j}E_3\bigg(E_2\big(e_2e_1\lambda_iE_3(e_3e_2e_1\lambda_jx)\big)e_3e_2e_1\lambda^*_j\bigg)e_3e_2e_1\lambda^*_i\\
 &= {\tau}^{-6} \sum_{i,j} E_2\big(e_2e_1\lambda_iE_3(e_3e_2e_1\lambda_jx)\big)E_3(e_3)e_2e_1\lambda^*_je_3e_2e_1\lambda^*_i\\
 &= {\tau}^{-5} \sum_{i,j} E_2\circ E_3\big(e_2e_1\lambda_ie_3e_2e_1\lambda_jx\big)e_2e_1\lambda^*_je_3e_2e_1\lambda^*_i\\
 &= \rho^{B\subset A_{1}}_1.
\end{align*}
This completes the proof.
\end{proof}

\begin{notation}
Following \cite{Oc} (see also \cite{Bi}), we denote the rotations
$\rho^{B\subset A}_{1}$ and $\rho^{B\subset A_{1}}_1$, respectively, by
${\gamma}_0$ and ${\gamma}_1$.
\end{notation}
Ocneanu called them
\textit{mirrorings} ({for the $II_1$-factor case}). These will prove to be very important in what
follows.

\begin{remark}\label{formulaforrotations2}
 Following Remark \ref{formulaforrotations} and the above notation, we
 see that
\begin{equation}
  {\gamma}_0(x)={\tau}^{-1} \sum_i E_1(e_1\lambda_ix)e_1\lambda^*_i,
   \end{equation}
 for every $x\in B^{\prime}\cap A_1$. Similarly, using Propositions \ref{iteratingbasicconstruction} and \ref{multibasis},  we have
 \begin{equation}
 \gamma_1(y)={\tau}^{-5}\sum_{i,j}E_2\circ E_3(e_2e_1e_3e_2\lambda_ie_1\lambda_jy)e_2e_1e_3e_2\lambda^*_je_1\lambda^*_i,
  \end{equation}
 for every $y\in B^{\prime}\cap A_3$.
\end{remark}

 We next show that $\gamma_0$ and $\gamma_1$ are both
 $*$-preserving anti-automorphisms.  This requires
some work. We break the proof into various steps.  

\begin{lemma}
 \label{r1}
 $\gamma_{k}$, for $k\in \{0,1\}$, is a $*$-preserving map.
\end{lemma}
\begin{proof}
We prove only for $k=0$. The proof for $\gamma_1$ will follow once we
apply the same technique for the inclusion $B\subset A_1$ with the
minimal conditional expectation $E_0\circ E_1$, since
$\gamma_1=\rho^{B\subset A_1}_1.$
\smallskip

\noindent Let $x \in B'\cap A_1$. Then, for $a,b\in A$, we have
 \begin{align*}
  \big \langle {\gamma}_0(x) (a) , b\big\rangle_B &=
       {\tau}^{-1}\big\langle \sum_i
       E_1(e_1\lambda_ix)E_0(\lambda^*_ia),b\big\rangle_B\\ &=
       {\tau}^{-1} \big\langle \sum_i
       E_1\big(e_1\lambda_ixE_0(\lambda^*_ia)\big),b\big\rangle_B\\ &=
       {\tau}^{-1}\big\langle E_1\bigg(e_1\big(\sum_i
       \lambda_iE_0(\lambda^*_ia)\big)x\bigg),b\big\rangle_B
       ~~~~~\hspace{12mm} (\textrm{since}~~x\in B^{\prime})\\ &=
              {\tau}^{-1}\big\langle
              E_1(e_1ax),b\big\rangle_B\\ &={\tau}^{-1}E_0\big(E_1(x^*a^*e_1)b\big) \end{align*}
and, on similar lines, 
\begin{align*}
 \big\langle a,\gamma_0(x^*)(b)\big\rangle_B 
 &= {\tau}^{-1}\big\langle a,\sum_iE_1(e_1\lambda_ix^*)e_1\lambda^*_ib\big\rangle_B\\
 &={\tau}^{-1} \big\langle a,E_1(e_1bx^*)\big\rangle_B\\
 &= {\tau}^{-1}E_0\big(a^*E_1(e_1bx^*)\big).
\end{align*}
Since $x \in B'\cap A_1$ and $E_0 \circ E_1: A_1 \rar B$ is the
minimal conditional expectation, thanks to \cite[Lemma 3.11]{KW}, we
have $ E_0\circ E_1(x^*a^*e_1b) = E_0\circ E_1(a^*e_1bx^*).$ Hence,
$\gamma_0(x)$ is adjointable with
${\big({\gamma}_0(x)\big)}^*=\gamma_0(x^*).$ \color{black}
\end{proof}

\begin{remark}\label{adjointoffourier}
 From Lemma \ref{r1}, it is obvious that
 $\big(\mathcal{F}_k(x)\big)^*=\mathcal{F}_k\big(\gamma_{k-1}(x^*)\big)$
 for $k\in \{1,2\}.$
\end{remark}

 \begin{lemma}
 \label{r2}
 ${{\gamma}}^2_{k}=\mathrm{Id}$ for $k\in \{0,1\}$.
\end{lemma}
\begin{proof}
Observe that for $x\in B^{\prime}\cap A_1$, we have
 \begin{align*}
  {{\gamma}_0}^2(x)
  &={\tau}^{-1} \sum_i E_1\big(e_1\lambda_i\gamma_0(x)\big)e_1\lambda^*_i\\
  &={\tau}^{-2}\sum_{i,j} E_1\big(e_1\lambda_iE_1(e_1\lambda_jx)e_1\lambda^*_j\big)e_1\lambda^*_i\\
  &={\tau}^{-2}\sum_{i,j} E_1\bigg(E_0\big(\lambda_iE_1(e_1\lambda_jx)\big)e_1\lambda^*_j\bigg)e_1\lambda^*_i\\
  &={\tau}^{-2} \sum_{i,j}E_1\bigg(E_0\big(\lambda_iE_1(e_1\lambda_jx)\big)e_1\bigg)\lambda^*_je_1\lambda^*_i\\
  &= {\tau}^{-2}\sum_{i,j} E_0\big(\lambda_iE_1(e_1\lambda_jx)\big)E_1(e_1)\lambda^*_je_1\lambda^*_i\\
  &={\tau}^{-1} \sum_{i,j} E_0\circ E_1(\lambda_ie_1\lambda_jx)\lambda^*_je_1\lambda^*_i\\
  &= {\tau}^{-1}\sum_{i,j} E_0\circ E_1(x\lambda_ie_1\lambda^*_j)\lambda^*_je_1\lambda^*_i    \hspace{8mm} \textrm{\big(by \cite[Lemma 3.11]{KW}\big)}\\
  &=x.
\end{align*}
In the last equality, we have used the fact  (\Cref{multibasis}) that the collection
$\{{\tau}^{-1/2}\lambda_ie_1\lambda_j:1\leq i,j\leq n\}$ is a quasi-basis for the
minimal conditional expectation $E_0\circ E_1$.
\smallskip

\noindent That ${\gamma}^2_1=\mathrm{Id}$ follows once we repeat the
same procedure as above using Propositions
\ref{iteratingbasicconstruction} and \ref{multibasis}. This completes
the proof.\end{proof}

\begin{lemma}
 \label{r3}
$\gamma_{k}$, for $k\in \{0,1\}$, is an anti-homomorphism.
\end{lemma}
\begin{proof}
 For $x,y\in B^{\prime}\cap A_1$, we have
 \begin{align*}
  \gamma_0(x)\gamma_0(y) &= {\tau}^{-2} \sum_{i,j}
  E_1(e_1\lambda_ix)e_1\lambda^*_iE_1(e_1\lambda_jy)e_1\lambda^*_j\\ &={\tau}^{-2}
  \sum_{i,j}
  E_1(e_1\lambda_ix)E_0\big(\lambda^*_iE_1(e_1\lambda_jy)\big)e_1\lambda^*_j\\ &=
  {\tau}^{-2} \sum_{i,j}
  E_1\bigg(e_1\lambda_ixE_0\big(\lambda^*_iE_1(e_1\lambda_jy)\big)\bigg)e_1\lambda^*_j\\ &=
  {\tau}^{-2} \sum_{i,j}
  E_1\bigg(e_1\lambda_iE_0\big(\lambda^*_iE_1(e_1\lambda_jy)\big)x\bigg)e_1\lambda^*_j \hspace*{25mm} \big(\text{since}
  ~x\in B^{\prime}\big)\\ &={\tau}^{-2}
  \sum_{j}E_1\big(e_1E_1(e_1\lambda_jy)x\big)e_1\lambda^*_j\\ &=
      {\tau}^{-1}
      \sum_jE_1(e_1\lambda_jyx)e_1\lambda^*_j~~~~\hspace{48mm}
      \big(\textrm{by
        Lemma}~~~\ref{pushdown}\big)\\ &=\gamma_0(yx).
 \end{align*}
 The proof for $\gamma_1$ is similar and is left to the reader. \end{proof}

We have thus proved the following:
 \begin{theorem}\label{antiauto}
The rotation map  $\gamma_{k}:B^{\prime}\cap A_{2k+1}\rightarrow
B^{\prime}\cap A_{2k+1}$, for $k\in \{0,1\}$, is a $*$-preserving
anti-automorphism.
 \end{theorem}

 When $B \subset A$ is irreducible, then it turns out that $\gamma_0$ preserves trace as well.
\begin{lemma}\label{gamma0istracepreserving}
If $B \subset A$ is irreducible, then $\gamma_0$ is a $\tr$-preserving
map on $B'\cap A_1$.
\end{lemma}
\begin{proof}
For $x\in B^{\prime}\cap A_1$, we have
 \begin{align*}
  \tr\big(\gamma_0(x)\big)
  &= {\tau}^{-1} E_0\circ E_1\bigg(\sum_iE_1(e_1\lambda_ix)e_1\lambda^*_i\bigg)\\
  &= \sum_iE_0\bigg(E_1(e_1\lambda_ix)\lambda^*_i\bigg)\\
  &= E_0\circ E_1\bigg(e_1\big(\sum_i\lambda_ix\lambda^*_i\big)\bigg).
 \end{align*}
 Thanks to \Cref{min4} we see that $\sum_i\lambda_ix\lambda^*_i\in
 A^{\prime}\cap A_1.$ Since $B^{\prime}\cap A=\C$, using
   \Cref{inverse} and \Cref{f2}, we see that
 $\sum_i\lambda_ix\lambda^*_i={\tau}^{-1}\tr(x).$ Thus,
 $\tr(\gamma_0(x))=\tr(x).$ This completes the proof.
\end{proof}

We can now talk about the shift operator $ \gamma_1\gamma_0$ from $B'\cap A_1$ onto $A'\cap A_3$.
\begin{theorem}\label{shift}
$\gamma_1\gamma_0$ is a $\mathrm{tr}$-preserving $*$-isomorphism from
  $B^{\prime}\cap A_1$ onto $A^{\prime}_1\cap A_3$ and its
    inverse is the map $\gamma_0 \gamma_1$.
\end{theorem}
\begin{proof}
 We first prove that $\gamma_1$ maps $ B^{\prime}\cap A_1$ into
 ${A_1}^{\prime}\cap A_3.$
 
   Suppose $\{\lambda_i:1 \leq i\leq n\}$ is a quasi-basis for $E_0$. Put
  $\lambda_{ij}={\tau}^{-\frac{1}{2}} \lambda_ie_1\lambda_j.$ Then, by
\Cref{multibasis}, it follows that $\{\lambda_{ij}:1\leq i,j\leq n
  \}$ is a quasi-basis for $E_0\circ E_1.$ Therefore, for any $y\in
  B^{\prime}\cap A_3$, by Remark \ref{formulaforrotations2}, we have
  $$\gamma_1(y)={\tau}^{-2}\sum_{i,j} E_2\circ
  E_3\big(e_{[-1,1]}\lambda_{ij}y\big)e_{[-1,1]}\lambda^*_{ij},$$
 where $e_{[-1, 1]}$ is as in
    \Cref{iteratingbasicconstruction}. In particular, for any $x\in
    B^{\prime}\cap A_1$, we obtain
 \begin{align*}
  \gamma_1(x) &={\tau}^{-2}\sum_{i,j} E_2\circ
  E_3\big(e_{[-1,1]}\lambda_{ij}x\big)e_{[-1,1]}\lambda^*_{ij}\\ &=
  {\tau}^{-2}\sum_{i,j} E_2\circ
  E_3\big(e_{[-1,1]}\big)\lambda_{ij}xe_{[-1,1]}\lambda^*_{ij}\\ &=\sum_{ij}\lambda_{ij}xe_{[-1,1]}\lambda^*_{ij}\\ &=
  {\tau}^{-2} E^{B^{\prime}\cap A_3}_{{A_1}^{\prime}\cap
    A_3}\big(xe_{[-1,1]}\big). \numberthis \label{importantequation1}~~~~~~~\hspace{25mm}\textrm{(by
    \Cref{f2}})
 \end{align*}
This proves that $\gamma_1(x)\in {A_1}^{\prime}\cap A_3$ for
all $x \in B'\cap A_1$.  Thus, by \Cref{antiauto}, $\gamma_1\gamma_0$
is an injective $*$-homomorphism from $B^{\prime}\cap A_1$ into
${A_1}^{\prime}\cap A_3$.\smallskip

In order to show that $\gamma_0 \gamma_1$ is the inverse of
  $\gamma_1 \gamma_0$, we first show that $\gamma_1$ maps
  ${{A_1}^{\prime}\cap A_3}$ into $ B^{\prime}\cap A_1$. Let $z \in
  A_1'\cap A_3$. Then,
\begin{align*}
 \gamma_1(z) &= {\tau}^{-5}\sum_{i,j}E_2\circ
 E_3(e_2e_1e_3e_2\lambda_ie_1\lambda_jz)e_2e_1e_3e_2\lambda^*_je_1\lambda^*_i\\ &=
 {\tau}^{-5}\sum_{i,j}E_2\circ
 E_3(e_2e_1e_3e_2z\lambda_ie_1\lambda_j)e_2e_1e_3e_2\lambda^*_je_1\lambda^*_i\\ &=
 {\tau}^{-5}\sum_{i,j}E_2\circ
 E_3(e_2e_1e_3e_2z)\lambda_ie_1\lambda_je_2e_1e_3e_2\lambda^*_je_1\lambda^*_i\\ &=
 {\tau}^{-5}\sum_{i,j}E_2\circ
 E_3(e_2e_1e_3e_2z)\lambda_ie_1e_2\lambda_je_1\lambda^*_je_3e_2e_1\lambda^*_i\\ &=
 {\tau}^{-5}\sum_{i}E_2\circ
 E_3(e_2e_1e_3e_2z)\lambda_ie_1e_2e_3e_2e_1\lambda^*_i
 ~~~~~~~~~\hspace{5mm}(\textrm{by \Cref{basis}})\\ &={\tau}^{-3}\sum_i
 E_2\circ E_3(e_2e_1e_3e_2z)\lambda_ie\lambda^*_i\\ &=
 {\tau}^{-3}E_2\circ E_3(e_2e_1e_3e_2z) \in B^{\prime}\cap A_1.
\end{align*}
Therefore, $\gamma_0\gamma_1: A^{\prime}_1\cap A_3
  \rightarrow B^{\prime}\cap A_1$ is a well-defined $*$-homomorphism
  and we have $(\gamma_1\gamma_0)(\gamma_0\gamma_1)=\mathrm{Id}$, by \Cref{r2}.  In
  particular, this proves that $\gamma_1\gamma_0$ and
  $\gamma_0\gamma_1$ are inverses of each other.

To see that $\gamma_1\gamma_0$ is $tr$-preserving we use
Eq.\eqref{importantequation1} to note that for any $x\in
B^{\prime}\cap A_1$, we have
\begin{align*}
 \mathrm{tr}\big(\gamma_1\gamma_0(x)\big)
 &=\mathrm{tr}\bigg({\tau}^{-2}E^{B^{\prime}\cap A_3}_{{A_1}^{\prime}\cap A_3}(xe_{[-1,1]})\bigg)\\
 &= {\tau}^{-2}\mathrm{tr}\big(xe_{[-1,1]}\big)\\
 &= \mathrm{tr}(x). ~~~~~\hspace{50mm}\textrm{(by \Cref{f1})}
\end{align*}
This completes the proof of the theorem.
\end{proof}

\subsection{Coproduct on the relative commutants} \(\)

Each higher relative commutant comes equipped with another product,
the so-called \textit{coproduct} (Ocneanu called it `convolution'),
as defined below.
\begin{definition}\label{coproduct}
The coproduct of any two elements $x$ and $y$ of $B^{\prime}\cap A_k$,
denoted by $ x \circ y$, is defined  as
\[
x\circ
y={\mathcal{F}}^{-1}_k\big(\mathcal{F}_k(y)\mathcal{F}_k(x)\big).
\]
\end{definition}

\begin{lemma}\label{associativity}
  The coproduct $`\circ$' is associative.
\end{lemma}
\begin{proof}
The proof is basically a simple book keeping exercise.  Consider $x,y$
and $z$ in $B^{\prime}\cap A_k$.  Let us agree to denote
$e_{k+1}e_k\cdots e_2e_1$ by $v_{k+1}.$ From the definition, it follows
that $ x\circ y$ equals \begin{equation*}
  {\tau}^{-\frac{3(k+2)}{2}}E_{k+1}\bigg(E^{B^{\prime}\cap
    A_{k+1}}_{A^{\prime}\cap A_{k+1}}\big(ye_{k+1}e_k\cdots
  e_2e_1\big) E^{B^{\prime}\cap A_{k+1}}_{A^{\prime}\cap
    A_{k+1}}(xe_{k+1}e_k\cdots e_2e_1\big)e_1e_2\cdots
  e_ke_{k+1}\bigg).
                 \end{equation*}
Therefore,
 \begin{eqnarray*}
\lefteqn{  E^{B^{\prime}\cap A_{k+1}}_{A^{\prime}\cap A_{k+1}}\bigg((x\circ
  y)v_{k+1}\bigg)} \\& = & {\tau}^{-\frac{3(k+2)}{2}}E^{B^{\prime}\cap
    A_{k+1}}_{A^{\prime}\cap
    A_{k+1}}\Bigg(E_{k+1}\bigg(E^{B^{\prime}\cap
    A_{k+1}}_{A^{\prime}\cap A_{k+1}}\big(ye_{k+1}e_k\cdots
  e_2e_1\big)\\ & & E^{B^{\prime}\cap A_{k+1}}_{A^{\prime}\cap
    A_{k+1}}(xe_{k+1}e_k\cdots e_2e_1\big)e_1e_2\cdots
  e_ke_{k+1}\bigg)e_{k+1}e_k\cdots e_2e_1\Bigg)\\ &= &
  {\tau}^{-\frac{3(k+2)}{2}}\tau E^{B^{\prime}\cap
    A_{k+1}}_{A^{\prime}\cap A_{k+1}}\bigg(E^{B^{\prime}\cap
    A_{k+1}}_{A^{\prime}\cap A_{k+1}}\big(ye_{k+1}e_k\cdots
  e_2e_1\big)\\ & & E^{B^{\prime}\cap A_{k+1}}_{A^{\prime}\cap
    A_{k+1}}(xe_{k+1}e_k\cdots e_2e_1\big) e_1e_2\cdots
  e_ke_{k+1}e_k\cdots e_2e_1\bigg) \text{\hspace*{10mm} (by
    \Cref{pushdown})}\\ &= & {\tau}^{-\frac{(k+2)}{2}} E^{B^{\prime}\cap
    A_{k+1}}_{A^{\prime}\cap A_{k+1}}(yv_{k+1})E^{B^{\prime}\cap
    A_{k+1}}_{A^{\prime}\cap A_{k+1}}(xv_{k+1}).  \text{\hspace*{30mm}
    (by Eq. \eqref{three})}
 \end{eqnarray*}
Thus,
\begin{align*}
 (x\circ y)\circ z &\\ &=
  {\tau}^{-\frac{3(k+2)}{2}}E_{k+1}\Bigg(E^{B^{\prime}\cap
    A_{k+1}}_{A^{\prime}\cap A_{k+1}}(zv_{k+1})E^{B^{\prime}\cap
    A_{k+1}}_{A^{\prime}\cap A_{k+1}}\big((x\circ
  y)v_{k+1}\big)v^*_{k+1}\Bigg)\\ &=
  {\tau}^{-2(k+2)}E_{k+1}\Bigg(E^{B^{\prime}\cap
    A_{k+1}}_{A^{\prime}\cap A_{k+1}}(zv_{k+1})E^{B^{\prime}\cap
    A_{k+1}}_{A^{\prime}\cap A_{k+1}}(yv_{k+1})E^{B^{\prime}\cap
    A_{k+1}}_{A^{\prime}\cap A_{k+1}}(xv_{k+1})v^*_{k+1}\Bigg)\\ &=
  {\tau}^{-\frac{3(k+2)}{2}}E_{k+1}\Bigg(E^{B^{\prime}\cap
    A_{k+1}}_{A^{\prime}\cap A_{k+1}}\big((y\circ
  z)v_{k+1}\big)E^{B^{\prime}\cap A_{k+1}}_{A^{\prime}\cap
    A_{k+1}}\big(xv_{k+1}\big)v^*_{k+1}\Bigg)\\ &= x\circ (y\circ z).
\end{align*}
This completes the proof of associativity.
\end{proof}
\noindent Below we determine the identity element with respect to the
coproduct.
\begin{proposition}
 For every $x\in B^{\prime}\cap A_k$, we have
 \[
 x\circ ({\tau}^{-k/2} e_1e_2\cdots e_k)= x =  ({\tau}^{-k/2}e_1e_2\cdots e_k)\circ x.
 \]
 \end{proposition}
\begin{proof}
By \Cref{three}, we have $(e_1e_2\cdots e_{k-1}e_k)(e_{k+1}e_k\cdots
e_2e_1)= {\tau}^k e_1$. So,
 \begin{equation}\label{Sanat}
 \mathcal{F}_k(e_1e_2\cdots
 e_k)={\tau}^{-\frac{k+2}{2}}{\tau}^kE^{B^{\prime}\cap
   A_{k+1}}_{A^{\prime}\cap A_{k+1}}(e_1)={\tau}^{k/2}.
 \end{equation}
\color{black} Therefore, $x\circ (e_1e_2\cdots
e_k)={\mathcal{F}}_k^{-1}\bigg(\mathcal{F}_k(e_1e_2\cdots
e_k)\mathcal{F}_k(x)\bigg)={\tau}^{k/2}x$ by \Cref{inverse}.
Similarly, $(e_1e_2\cdots e_k)\circ x={\tau}^{k/2} x.$ This completes
the proof.
\end{proof}
\color{black}

 \section{Biprojections and intermediate $C^*$-subalgebras}

 In this section, we show how we can apply the results of the previous
 sections to understand the intermediate $C^*$-subalgebras of an
 irreducible inclusion of simple unital $C^*$-algebras. When the
 inclusion is a subfactor of $II_1$, a result by Bisch (\cite{Bi2})
 shows that the intermediate subfactors are in bijective
 correspondence with certain projections in the relative commutant,
 the so-called {\it biprojections}. Later, Bisch, Jones and Landau
 found a nice pictorial description of the {\it biprojections} in the
 planar algebraic language (see \cite{La}).  We obtain similar results
 for irreducible inclusions of simple unital $C^*$-algebras and we
 obtain an analogue of the Bisch's characterization.\smallskip

Throughout this section, $B \subset A$ will denote a fixed irreducible
pair (i.e., $B^{\prime}\cap A=\C$) of simple unital $C^*$-algebras
such that $\mathcal{E}_0(A,B) \neq \emptyset$.

\subsection{Intermediate $C^*$-subalgebras}\( \)

\begin{notation}\label{A-C-B-notation}
Let $C$ be an intermediate $C^*$-subalgebra of $B \subset A$. Clearly,
$C$ is also  simple  and there exists a unique
(minimal) condition $E^A_C$ from $A$ onto $C$.
  \begin{enumerate}
\item We denote the $C^*$-basic construction of the irreducible pair
  $C\subset A$ by $C_1$ with Jones projection $e_C$ corresponding to
  the minimal conditional expectation $E_C^A$, i.e., $C_1= C^*\la C,
  e_C\ra$.
 \item  We denote   ${[A:C]}^{-1}_0$ by ${\tau}_C$; so that 
  ${[C:B]}^{-1}_0=\tau/{\tau}_C.$
    
\end{enumerate}
  \end{notation}

We first provide few useful observations which will be used ahead.
\begin{lemma}\label{A-C-B}
Let $B \subset C \subset A$ be as in \Cref{A-C-B-notation}.  Then, we
have the following:
  \begin{enumerate}
\item $C_1$ is simple and unital.
  \item $C_1 \subset A_1$ is an irreducible pair with common identity.
\item   The unique (minimal) conditional expectations $E^{A_1}_A$ and $E^{C_1}_A$ satisfy 
  ${E^{A_1}_{A}}_{|_{C_1}} = E^{C_1}_{A}$. 
    \item ${E^{A_1}_{C_1}}_{|_{B'\cap A_1}}$ is the unique
      $\mathrm{tr}$-preserving conditional expectation from $B'\cap
      A_1$ onto $B'\cap C_1$.
  \item The tracial state on $C^{\prime}\cap C_1$
induced by the inclusion $C \subset A \subset C_1$ (as in \Cref{f1}) is same as the
restriction of the tracial state of $B^{\prime}\cap A_1.$
\end{enumerate}
  \end{lemma}

  \begin{proof}
    (1) follows from \Cref{dual-ce}. \smallskip

    (2): First note that
   $e_C\circ e_B=e_B$. Indeed,  for an arbitrary
   $a\in A$ we have $e_C\circ
    e_B(a)=e_C\big(E^A_B(a)\big)=E^A_C\big(E^A_B(a)\big)=E^A_B(a)=e_B(a)$.

    Now, recall from \cite[Proposition 1.6.6]{Wa} that
    $C_1=\text{span}\{x_1e_Cx_2:x_1,x_2\in A\}$ and $
    A_1=\text{span}\{y_1e_By_2:y_1,y_2\in A\}.$ Since $A_1$ is unital, it suffices to show  that
    $C_1 A_1\subset A_1$, which  is rather
    trivial as 
    $(x_1e_Cx_2)(y_1e_By_2)=x_1e_Cx_2y_1e_Ce_By_2=x_1E_C(x_2y_1)e_By_2\in
    A_1$ for all  $x_i, y_i\in A$, $i = 1, 2$. \smallskip

    (3) is now immediate from \Cref{intermediate-ce}
    because $A \subset A_1$ is also an irreducible pair of simple
    unital $C^*$-algebras, by \Cref{inverse}.\smallskip

    (4):  Let  $x\in
   B^{\prime}\cap A_1$. Clearly, $E^{A_1}_{C_1}(x)\in
   B^{\prime}\cap C_1$ and, by (3), we have  $E^{A_1}_A\circ E^{A_1}_{C_1}(xy)=E^{C_1}_A\circ
   E^{A_1}_{C_1}(xy)$ for every $y \in B'\cap C_1$. Hence,
   \[
   \tr\big(E^{A_1}_{C_1}(x)y\big)=E^A_B\circ E^{C_1}_A\circ
   E^{A_1}_{C_1}(xy)= E^A_B\circ E^{A_1}_A(xy)=\tr(xy)
   \]
   for every $y \in B'\cap C_1$.
       \smallskip

(5): It suffices to show that
       \[
       E^A_B\circ E^{A_1}_A|_{C^{\prime}\cap C_1}= E^A_C\circ
       E^{C_1}_A|_{C^{\prime}\cap C_1}.
       \]
       For any $z\in C^{\prime}\cap
   C_1$, we see that $E^A_B\circ E^{A_1}_A(z)=E^C_B\circ E^A_C\circ
   E^{C_1}_A(z)$, by (3). But $E^{C_1}_A(z)\in C^{\prime}\cap A$ and so
   $E^A_C\circ E^{C_1}_A(z)\in \mathcal{Z}(C)=\C$.  Therefore,
   $E^A_B\circ E^{A_1}_A(z)=E^C_B\big(E^A_C\circ
   E^{C_1}_A(z)\big)=E^A_C\circ E^{C_1}_A(z).$ This completes the
   proof.
  \end{proof}
\color{black}
  \begin{notation}\label{C-2-C-3}
    Let $A \subset C \subset B$ be as in \Cref{A-C-B}. \begin{enumerate}
\item Let $C_2$ denote
the $C^*$-basic construction of $C_1 \subset A_1$ with Jones
projection denoted by $e_{C_1}$ corresponding to the unique minimal
conditional expectation $E^{A_1}_{C_1}$.
    \item  As in \Cref{A-C-B}, $C_2 \subset A_2$ is a simple unital irreducible
    inclusion. Let $C_3$ denote its $C^*$-basic construction with
    Jones projection denoted by $e_{C_2}$.
\end{enumerate}
    \end{notation}

\begin{lemma} \label{gamma0ec}\label{e1 and ec}
   Let $B\subset C \subset A$ be  as in \Cref{A-C-B}. Then,
\begin{enumerate}
\item  $\gamma_0(e_C)=e_C.$ In particular,
  $\gamma_0(e_1)=e_1.$
\item  $E^{A_1}_{C_1}(e_1)=\frac{1}{[C:B]}_0 e_C.$
\end{enumerate}
\end{lemma}

\begin{proof}
(1): Observe that $e_1e_C=e_Ce_1=e_1.$ Fix a quasi-basis
    $\{\lambda_i\}$ for $E_0$. Then,
 \begin{align*}
  \gamma_0(e_C) &= {\tau}^{-1}\sum_i
  E_1(e_1\lambda_ie_C)e_1\lambda^*_i\\  &= {\tau}^{-1} \sum_i
  E_1(e_1e_C\lambda_ie_C)e_1\lambda^*_i\\ &=
  {\tau}^{-1} \sum_i
  E_1\big(e_1E^A_C(\lambda_i)e_C\big)e_1\lambda^*_i\\ &=\sum_iE^A_C(\lambda_i)e_1\lambda^*_i\\ &=\sum_iE^A_C(\lambda_i)e_Ce_1\lambda^*_i\\ &=
  \sum_i
  e_C\lambda_ie_Ce_1\lambda^*_i\\ &=e_C\big(\sum_i\lambda_ie_1\lambda^*_i)\\ &=
  e_C. \qquad \qquad \qquad \qquad \qquad \text{(by \Cref{support})}
 \end{align*}

 (2): Let $w = \sum_i x_i e_C y_i \in C'\cap C_1$ with $x_i, y_i \in
 A$ for all $i$. By \Cref{A-C-B}, the tracial state on $C'\cap C_1$ is the
 restriction of the tracial state on $B'\cap A_1$. So, we have
 \begin{align*}
 \mathrm{tr}(e_C w) &= \mathrm{tr}\big(\sum_i e_C x_ie_Cy_i\big)
 \\ &= E^A_B\circ E^{A_1}_A\big(\sum_i e_CE^A_C(x_i)y_i\big)\\ &=
    {[A:C]}^{-1}_0E^A_B\big(\sum_iE^A_C(x_i)y_i\big). \qquad \qquad
    \textrm{(since } E^{A_1}_A(e_C)={[A:C]}^{-1}_0)
  \end{align*}
And, on the other hand,
\begin{align*}
 \mathrm{tr}(e_1w) &= \mathrm{tr}(\sum_i e_1x_ie_Cy_i)\\ &=
 E^A_B\big(\sum_i e_1E^A_C(x_i)y_i\big)\qquad \qquad \qquad (\text{since } e_1e_C =
 e_C e_1 = e_1)\\ &= {\tau}  E^A_B \big(\sum_i
 E^A_C(x_i)y_i\big).
\end{align*}
Thus, $\mathrm{tr}(e_Cw)={[C:B]}_0 \mathrm{tr}(e_1w)$ for all $w \in
B'\cap C_1$; so that $E^{B'\cap A_1}_{C'\cap
  C_1}(e_1)=\frac{1}{[C:B]}_0 e_C.$ Then, by \Cref{A-C-B}, we deduce
that $E^{A_1}_{C_1}(e_1)=\frac{1}{[C:B]}_0 e_C.$
\end{proof}

A more general version  of the following lemma will be established in the next section.
\begin{lemma}\label{ece2ec}
Let $B\subset C \subset A$ be  as in \Cref{A-C-B}.   Then, 
\[
e_Ce_2e_C={[A:C]}^{-1}_0 e_Ce_{C_1} \text{ and } e_{C_1}e_3e_{C_1}={[C:B]}^{-1}_0 e_{C_1}e_{C_2}.
\]
\end{lemma}

\begin{proof}
This proof is essentially borrowed from \cite{Bi2,BJ,GJ}. We first show that
 \begin{equation}
  \label{traceofec}
   \mathrm{tr}(e_C)={\tau}_c ~~~\text{and}~~~    \mathrm{tr}(e_Ce_{C_1})=\tau.
\end{equation}
\color{black}
 The first equality follows trivially from the definition. To see the second equation note that 
\begin{align*}
    \mathrm{tr}(e_Ce_{C_1})
 & =E^A_C\circ E^{C_1}_A\circ E^{A_1}_{C_1}\circ E^{C_2}_{A_1}(e_Ce_{C_1})\\
 &= E^A_C\circ E^{C_1}_A\bigg(e_C E^{A_1}_{C_1}\circ E^{C_2}_{A_1}\big(e_{C_1}\big)\bigg)\\
 &= E^A_C\circ E^{C_1}_A\bigg(e_C{{[A_1:C_1]}_0}^{-1}\bigg)\\
 &=\tau,
 \end{align*}
where the last equality follows from the derivation
\[
  {[A_1:C_1]}_0=\frac{{[A_1:A]}_0}{{[C_1:A]}_0}=\frac{{[A:B]}_0}{{[A:C]}_0}={[C:B]}_0.\quad  (\text{by Theorem } \ref{min3})
  \]

  Now, $e_2e_Ce_2=E^{A_1}_A(e_C)e_2={\tau}_ce_2.$ So,
  $v:=\frac{1}{\sqrt{{\tau}_c}}e_2e_C$ is a partial isometry and hence
  $v^*v=\frac{1}{{\tau}_c}e_Ce_2e_C$ is a projection, say,
  $q$. Clearly, $q$ majorizes the projection $e_Ce_{C_1}$. And, on the
  other hand, by \Cref{f1}, we have $\tr(q)=\frac{1}{{\tau}_c}
  \tr(e_Ce_2)=\frac{\tau}{{\tau}_c} tr(e_C)=\tau.$ Therefore, by
  \Cref{traceofec}, $q =e_Ce_{C_1}.$ This proves that
  $e_Ce_2e_C={[A:C]}^{-1}_0 e_Ce_{C_1}$.  The  other
  implication follows similarly and we omit the details.
\end{proof}

\begin{proposition} \label{fouriertransformofec1}
  Let $B \subset C \subset A$ be as in \Cref{A-C-B}. Then,
 \begin{equation}\label{f-ec}
   \mathcal{F}_1(e_C)=\displaystyle \frac{\sqrt{{[A:B]}_0}}{{[A:C]}_0}
   e_{C_1}.
 \end{equation}
 In particular, $\mathcal{F}_1(e_1)=\frac{1}{\sqrt{{[A:B]}_0}} 1_{A_1}.$
\end{proposition}

\begin{proof}
 This follows immediately from the preceding lemma as follows:
 \begin{align*}
  {\mathcal{F}}_1(e_C) &= {\tau}^{-3/2} E^{B^{\prime}\cap
    A_2}_{A^{\prime}\cap A_2}\big(e_Ce_2e_1\big)\\ &={\tau}^{-3/2}
  E^{B^{\prime}\cap A_2}_{A^{\prime}\cap
    A_2}\big(e_Ce_2e_Ce_1\big)\\ &= {\tau}^{-3/2} E^{B^{\prime}\cap
    A_2}_{A^{\prime}\cap A_2}\big({\tau}_ce_{C_1}e_1\big)
  \text{\hspace*{23mm} (by \Cref{ece2ec})}\\ &={\tau}_c{\tau}^{-1/2}
  e_{C_1}. \text{\hspace*{40mm} (by \Cref{f2})}
 \end{align*}
\end{proof}

\begin{remark} \label{fouriertransformofec2}
Similar to \Cref{fouriertransformofec1}, we also have the following expression:
\[
\mathcal{F}^{A\subset A_1}_1(e_{C_1})=\displaystyle
\frac{\sqrt{{[A:B]}_0}}{{[C:B]}_0} e_{C_2},
\]
where $\mathcal{F}^{A\subset A_1}_1$ is the Fourier transform from
$A^{\prime}\cap A_2$ onto $A^{\prime}_1\cap A_3$ defined by
$$ \mathcal{F}^{A\subset A_1}_1(x)={\tau}^{-3/2}E^{A^{\prime}\cap
  A_3}_{A^{\prime}_1\cap A_3}(xe_3e_2).$$
\end{remark}

\bigskip

\noindent The map $\gamma_1\gamma_0$ in \Cref{shift} has the following
special property.
\begin{proposition}
 \label{gamma1gamma0ec}
 $\gamma_1\gamma_0(e_C)=e_{C_2}.$
\end{proposition}

\begin{proof}
First, note that, $\gamma_1\gamma_0(e_C)=\gamma_1(e_C)$, by \Cref{gamma0ec}. And then, we have
 \begin{align*}
  \gamma_1(e_C)
  &= {\tau}^{-2} E^{B^{\prime}\cap A_3}_{A^{\prime}_1\cap A_3}\big(e_Ce_{[-1,1]}\big)\\
  &={\tau}^{-3}E^{B^{\prime}\cap A_3}_{A^{\prime}_1\cap A_3}\big(e_Ce_2e_C.e_1e_3e_2\big)\\
  &= {\tau}^{-3}{\tau}_c E^{B^{\prime}\cap A_3}_{A^{\prime}_1\cap A_3}\big(e_{C_1}e_1e_3e_2\big) \text{\hspace*{30mm} (by \Cref{ece2ec})}\\
  &= {\tau}_c{\tau}^{-1}\sum_{i,j} \lambda_ie_1\lambda_je_{C_1}e_1e_3e_2\lambda^*_je_1\lambda^*_i \text{\hspace*{18mm} (by \Cref{f2})}\\
  &={\tau}_c{\tau}^{-2}\sum_{i}\lambda_ie_1e_{C_1}\big(\sum_j \lambda_je_1\lambda^*_j\big)e_3e_2e_1\lambda^*_i\\
  &={\tau}_c{\tau}^{-2}\sum_i\lambda_ie_1e_{C_1}e_3e_{C_2}e_2e_1\lambda^*_i\\
  &=e_{C_2}. \text{\hspace*{60mm} (by \Cref{ece2ec} again)}
 \end{align*}
This finishes the proof.
 \end{proof}

The following Proposition will be used in the last section.
 \begin{proposition}\label{useful}   Let $B\subset C \subset A$ be  as in \Cref{A-C-B}.  Then, 
   \[
   \gamma_0(B^{\prime}\cap C_1)=C^{\prime}\cap A_1 \text{ and } \gamma_0(C^{\prime}\cap A_1)=B^{\prime}\cap C_1.
   \]
 \end{proposition}

 \begin{proof}
 Consider $x\in C^{\prime}\cap A_1$. Let $\{\gamma_j:j\in J\}$ be a
 quasi-basis for $E^A_C$ and $\{\lambda_i:i\in I\}$ be a quasi-basis
 for $E^A_B$. Then,
 \begin{align*}
  \gamma_0(x) & = {\tau}^{-1} \sum_{i,j}
  E^{A_1}_A\big(e_1\gamma^*_jE^A_C(\gamma_j\lambda_i)x\big)e_1\lambda^*_i\\ &=
  {\tau}^{-1}
  \sum_{i,j}E^{A_1}_A(e_1\gamma^*_jx)E^A_C(\gamma_j\lambda_i)e_Ce_1\lambda^*_i\\ &=
      {\tau}^{-1} \sum_j E^{A_1}_A(e_1\gamma^*_jx)e_C\gamma_j.\qquad
      \qquad \qquad (\textrm{since}~~ \sum_i\lambda_ie_1\lambda^*_i=1)
 \end{align*}
Therefore, $\gamma_0(x)\in C_1.$ In other
words,
\[
\gamma_0(C^{\prime}\cap A_1) \subseteq B^{\prime}\cap
C_1 \hspace{50mm} (\star).
\]
Next, consider $y\in B^{\prime}\cap C_1.$
We show that $\gamma_0(y)\in C^{\prime}\cap A_1.$ To see this note
that since for any $u\in \mathcal{U}(C)$ (the set of all unitaries of
$C$) we have $\{u\lambda_i:i\in I\}$ is a quasi-basis for $E^B_A$ and
hence
\begin{align*}
 \gamma_0(y) 
 &= {\tau}^{-1} \sum_i E^{A_1}_A(e_1u\lambda_iy)e_1\lambda^*_iu^*\\
 &= {\tau}^{-1} \sum_i E^{C_1}_A\circ E^{A_1}_{C_1}(e_1u\lambda_iy)e_1\lambda^*_iu^*\\
 &= {\tau}^{-1}\sum_i E^{C_1}_A\big(E^{A_1}_{C_1}(e_1)u\lambda_iy\big)e_1\lambda^*_iu^*\\
 &= {\tau}^{-1}\frac{1}{[C:B]}_0\sum_i E^{C_1}_{A}(e_Cu\lambda_iy)e_1\lambda^*u^* \text{\hspace*{25mm} (by \Cref{e1 and ec})}\\
 &= {\tau}^{-1} \frac{1}{{[C:B]}_0}u \bigg(\sum_i E^{C_1}_A(e_C\lambda_iy)e_1\lambda^*_i\bigg)u^*\\
 &= {\tau}^{-1} 
 u\bigg(\sum_i E^{C_1}_A\big(E^{A_1}_{C_1}(e_1)\lambda_iy\big)e_1\lambda^*_i\bigg)u^*\\
 &= u\bigg({\tau}^{-1} \sum_i E^{A_1}_A(e_1\lambda_iy)e_1\lambda^*_i\bigg)u^*\\
 &= u\gamma_0(y)u^*.
\end{align*}
Therefore, $\gamma_0(y)\in C^{\prime}.$ In other words, we have proved that,
$$\gamma_0(B^{\prime}\cap C_1)\subseteq C^{\prime}\cap
A_1\hspace{49mm} (\star \star).$$ Combining $(\star)$ and $(\star
\star)$ together with \Cref{r2} we establish the desideratum.
\end{proof}

 \subsection{Biunitaries, biprojections  and  bipartial isometries  }\( \)
 
 Motivated by \cite{Oc,Bi2,JLW, Jo2}, we propose the following definitions. The notion of `bipartial isometry' was introduced and studied effectively in \cite{JLW} 
 for subfactors.
\begin{definition}\label{bipro}
  \begin{enumerate}
    \item A unitary $u\in B^{\prime}\cap A_k$ will be called a
      \textit{biunitary} for the inclusion $B\subset A$ if
      $\mathcal{F}_k(u)$ is again a unitary in $A'\cap A_{k+1}$.  We
      denote the collection of all biunitaries by
      ${\text{BU}}_k(B,A).$
  \item A projection $e\in B^{\prime}\cap A_k$ will be  called
    \textit{biprojection} for the inclusion $B\subset A$ if 
    $\mathcal{F}_k(e)$ is a multiple of a projection.    We denote
   the collection of all  biprojections by ${\text{BP}}_k(B,A).$
\item An element $e\in B^{\prime}\cap A_k$ will be called a
  \textit{bipartial isometry} for the inclusion $B\subset A$ if both
  $e$ and $\mathcal{F}_k(e)$ are multiples of partial isometries. We
  denote the collection of all  bipartial isometries by ${\text{BPI}}_k(B,A).$

  \end{enumerate}
\end{definition}

\begin{notation}
 Suppose $e\in \text{BP}_k(B,A)$. So, $\mathcal{F}_k(e)= t f$ for some
  positive  scalar $t$ and a projection $f\in A^{\prime}\cap A_{k+1}.$ We shall
 denote the projection $f$ by the symbol $[\mathcal{F}_k(e)].$
\end{notation}

\begin{remark}\label{bibiprojection1} By Proposition
 \ref{fouriertransformofec1} and Corollary
 \ref{fouriertransformofec2},  $e_C$ and $e_{C_1}$ are
both biprojections. More precisely, $e_C\in \mathrm{BP}_1(B,A)$ and
 $e_{C_1}\in \mathrm{BP}_1(A,A_1).$
\end{remark}

\begin{proposition}
 Fourier transform of a biunitary is a bipartial isometry. More
 precisely, for every $k \geq 1$,  \[
  \mathcal{F}_k(u)\in
 \mathrm{BPI}_{k+1}(B,A)\text{ for all } u\in \mathrm{BU}_k(B,A).
 \]
\end{proposition}

\begin{proof}
 First, notice that, being a unitary, $\mathcal{F}_k(u)$ is also a
 partial isometry in $A^{\prime}\cap A_{k+1}\subset B^{\prime}\cap
 A_{k+1}$. Next, we show that
 $\mathcal{F}_{k+1}\big(\mathcal{F}_k(u)\big)$ is a partial isometry
 as well. We have
 \begin{align*}
  \mathcal{F}_{k+1}\big(\mathcal{F}_k(u)\big)
  &= {\tau}^{-\frac{k+3}{2}} E^{B^{\prime}\cap A_{k+2}}_{A^{\prime}\cap A_{k+2}}\bigg(\mathcal{F}_k(u)e_{k+2}e_{k+1}\cdots e_2e_1\bigg)\\
  &= \tau {\tau}^{-\frac{k+3}{2}} \mathcal{F}_k(u) e_{k+2}e_{k+1}\cdots e_3e_2.
  \end{align*}
 Let $w:= \mathcal{F}_{k+1}\big(\mathcal{F}_k(u)\big).$ Then,
 \[
 ww^*=
{\tau}^2{\tau}^{-(k+3)} \mathcal{F}_k(u)(e_{k+2}e_{k+1}\cdots
e_3e_2)(e_2e_3\cdots e_{k+1}e_{k+2}){\mathcal{F}_k(u)}^*.
\]
Clearly, $(e_{k+2}e_{k+1}\cdots e_3e_2)(e_2e_3\cdots
e_{k+1}e_{k+2})={\tau}^k e_{k+2}.$ So,
$ww^*={\tau}^{-1}\mathcal{F}_k(u)e_{k+2}{\mathcal{F}_k(u)}^*$ and
since $\mathcal{F}_k(u)$ is a unitary it follows that $w$ is a
multiple of a partial isometry. This completes the proof.
\end{proof}

\begin{lemma}\label{e1iscentral}
$e_1$ is a minimal as well as a central projection in $B'\cap A_1$.
\end{lemma}
\begin{proof}
We first assert that  $e_1$ is a minimal projection in $B'\cap A_1$.

Let $u \in \mathcal{U}(B'\cap A_1)$. By \Cref{pushdown}, we have $ue_1
= [A:B]_0 E_0(ue_1)e_N$.  Since $u, e_1 \in B'\cap A_1$, 
$E_B^A(u e_1) \in B' \cap A =\C$. Thus $e_1 u e_1 \in \C e_1$. So, $e_1$
is minimal  in $B'\cap A_1$.

Next, we show that $e_1$ is central as well.  Let $\lambda_0:=[A:B]_0
E_0(ue_1)\in \C$. We now show that $|\lambda_0| = 1$.   We have $ue_1 u^* =
\lambda_0 e_1 \bar{\lambda_0}$.   Applying $E_1$ on both sides, we get
 \[
   [A:B]_0^{-1} = \tr(e_1) = \tr(u e_1 u^*) = E_1(ue_1u^*) =
   |\lambda_0|^2 E_1(e_1) = [A:B]_0^{-1} |\lambda_0|^2.
   \]
   Hence, $ |\lambda_0| = 1$ and we get $ue_1 u^* = e_1$. Since $u$
   was an arbitrary unitary in $B'\cap A_1$, we deduce that $e_1 \in
   \mathcal{Z}(B'\cap A_1)$.
\end{proof}

\begin{lemma}
 \label{ee1}  \label{scalar1}  \label{scalar2}  \label{fe1e2}
 Let $e\in \mathrm{BP}_1(B,A)$ and $\mathcal{F}_1(e)=tf$ for some
 $t>0$ and projection $f\in A'\cap A_2$. Then,
 \begin{enumerate}
\item   \(
 ee_1=e_1e=e_1\) and \(fe_2=e_2f=e_2\);
 \item \(
   E_1(e)=\mathrm{tr}(e)= t{\tau}^{1/2}\) and \( E_2(f)=\mathrm{tr}(f)=t^{-1}{\tau}^{1/2}\);
 \item \(\tr(ef)=\tau
 \); and
\item  $fe_1e_2=t^{-1}{\tau}^{1/2}ee_2.$
 \end{enumerate}
\end{lemma}

\begin{proof}
  (1): We first assert that $ee_1=e_1.$ To see this, use Lemma
  \ref{pushdown} to obtain $ee_1={\tau}^{-1}E_1(ee_1)e_1= se_1$ for
  some scalar $s>0.$ Indeed, $E_1(ee_1)\in B^{\prime}\cap A=\C$ and
  hence $E_1(ee_1)=E_0\circ E_1(ee_1)=\tr(e_1).$ And, by Lemma
  \ref{e1iscentral}, $ee_1$ is a projection. Therefore, $s=1.$ The
  assertion about $f$ follows similarly.
\smallskip

(2): We first show that $E_1(e) = \mathrm{tr}(e)$. Note that  $e\in
B^{\prime}\cap A_1$ and  hence, for any $b\in B$,
$E_1(e)b=E_1(eb)=E_1(be)=bE_1(e).$ Therefore, $E_1(e)\in
B^{\prime}\cap A=\C$.  Thus, $E_1(e)=E_0\circ E_1(e)=\tr(e).$
Similarly, we obtain $E_1(f) = \tr(f)$.

Further, from the definition of $f$,  we obtain $e=
  {\tau}^{-\frac{3}{2}}E_2(fe_1e_2).$  Thus, $\tr(e)=t\,
   {\tau}^{-\frac{3}{2}} E_0\circ E_1\circ E_2(fe_1e_2).$ Then, by
   \cite[Lemma 3.11]{KW}, we get
   \[
   \tr(e)=t\,
        {\tau}^{-\frac{3}{2}}E_0\circ E_1\circ E_2(e_2e_1)=t\,
        {\tau}^{-\frac{3}{2}}{\tau}^2=t{\tau}^{\frac{1}{2}}.
        \]
   For $\mathrm{tr}(f)$, observe that\vspace*{-5mm}
 \begin{align*}
  \tr(f)
  & =t^{-1}{\tau}^{-\frac{3}{2}} \tr\big(E^{B^{\prime}\cap A_2}_{A^{\prime}\cap A_2}(ee_2e_1)\big)\\
  &= t^{-1}{\tau}^{-\frac{3}{2}} \tr(ee_2e_1)\\
  &= t^{-1} {\tau}^{-\frac{3}{2}} \tr(e_1e_2)\\
  &= t^{-1}{\tau}^{\frac{1}{2}}.
 \end{align*}

 (3): \vspace*{-7mm} \begin{align*}
          \tr(ef) 
          &= E_0\circ E_1\circ E_2(ef)\\
          &= E_0\circ E_1\big(eE_2(f)\big)\\
          &= E_0\circ E_1(e)\, \tr(f)\\
          &=\tr(e)\, \tr(f)\\
          &= (t{\tau}^{\frac{1}{2}})(t^{-1}{\tau}^{\frac{1}{2}})\\
          &= \tau.
         \end{align*}

 (4): \vspace*{-7mm}  \begin{align*}
  fe_1e_2 &= t^{-1}{\tau}^{-\frac{3}{2}}E^{B^{\prime}\cap
    A_2}_{A^{\prime}\cap A_2}(ee_2e_1)e_1e_2\\ &=
  t^{-1}{\tau}^{-\frac{3}{2}}\tau \sum_i
  \lambda_iee_2e_1\lambda^*_ie_1e_2\\ &=t^{-1}{\tau}^{-\frac{1}{2}}\sum_i\lambda_ieE_1\big(E_0(\lambda^*_i)e_1\big)e_2\\ &=t^{-1}{\tau}^{\frac{1}{2}}
  \sum_i \lambda_iE_0(\lambda^*_i)ee_2 \text{ \hspace*{20mm}(since $e\in
    B^{\prime}$)}\\ &=t^{-1}{\tau}^{\frac{1}{2}}ee_2.
 \end{align*}\vspace*{-5mm}
\end{proof}

\noindent Using above observations, we deduce the following
\textit{exchange relation}:
\begin{proposition}\label{exchangerelation}
 Let $e$, $t$ and $f$ be as in \Cref{ee1}. Then, $ef=fe$.
\end{proposition}

\begin{proof}
 We first show that $\tr(fee_2e)=\tau \tr(e).$

 By \Cref{fe1e2}(4), we see that $ee_2e=t^2fe_1f$. Then, by other
 parts of the same Lemma, we obtain
 \[
 \tr(fee_2e)=t^2\tr(fe_1)=t^2E_0\circ E_1\big(E_2(f)e_1)
 = t^2 \tr(f)E_0\circ
 E_1(e_1)=\tau t^2 \tr(f) = \tau \tr(e).
 \]
Now, note that $e_2ee_2=E_1(e)e_2$ and hence, by \Cref{scalar1}(2), we
have $e_2ee_2=\tr(e) e_2.$ Further,\vspace*{-2mm}
 \begin{align*}
  {\lVert ee_2e - \tr(e)ef\rVert}^2_2 &=
  \tr\bigg(\big(ee_2e-\tr(e)fe\big)\big(ee_2e-\tr(e)ef\big)\bigg)\\ &=
  \tr\big(ee_2ee_2)- 2\tr(e)\,  \tr(fee_2e)+ \tr(e)^2
  \tr(ef)\\ &= \tau\, \tr(e)^2-2\tau \tr(e)^2+
  \tau\, \tr(e)^2 \text{\hspace*{20mm} (by
    \Cref{scalar2}(3))}\\ &=0;
 \end{align*}
so that, $ee_2e = \tr(e)ef$, and after taking adjoint we
obtain \begin{equation}\label{ef=fe} ef=\frac{ee_2e}{\tr(e)}=fe.
\end{equation}
\end{proof}

\begin{theorem}\label{efisabiprojection}
Let $e$, $t$ and $f$ be as in \Cref{ee1}. Then, $ef\in
\mathrm{BP}_2(B,A).$
\end{theorem}
\begin{proof}
 By \Cref{exchangerelation}, $ef$ is a projection and
 \begin{align*}
  \mathcal{F}_2(ef) & =   \mathcal{F}_2(fe) \\
  &= {\tau}^{-2} E^{B^{\prime}\cap A_3}_{A^{\prime}\cap A_3}\big(fee_3e_2e_1\big)\\
  &= {\tau}^{-2}f E^{B^{\prime}\cap A_3}_{A^{\prime}\cap A_3}\big(ee_3e_2e_1\big)\\
  &= {\tau}^{-2}f E^{B^{\prime}\cap A_3}_{A^{\prime}\cap A_3}\big(e_3ee_2e_1\big)\\
  &={\tau}^{-\frac{1}{2}}fe_3\mathcal{F}_1(e)\\
  &=t{\tau}^{-\frac{1}{2}}fe_3f\\
  &=\frac{1}{\tr(f)} fe_3f.
 \end{align*}
 Then,  by \Cref{scalar1}(2), we have
 $e_3fe_3=\tr(f)e_3$. Thus, $v:=\frac{1}{\sqrt{\tr(f)}}e_3f$ is a
 partial isometry with $vv^*=e_3$. Hence, $v^*v=\frac{1}{\tr(f)}fe_3f$
 is a projection too. In particular, $ef$ and $\mathcal{F}_2(ef)$ are both
 projections. This completes the proof.
\end{proof}
\color{black}

The following shows that, upto a scalar, a biprojection is also an
idempotent with respect to the coproduct.
\begin{lemma}\label{coprojection}
 If $e\in \mathrm{BP}_k(B,A)$, then $e\circ e=t e$ for some scalar $t$.
\end{lemma}
\begin{proof}
 By definition.
\end{proof}

\noindent Now we give an analogue of \Cref{coprojection} for biunitary element.
 \begin{lemma}
  $u\circ {\big(\gamma_0(u^*)\big)}={\big(\gamma_0(u^*)\big)}\circ u
   ={\tau}^{-1/2}e_1.$
 \end{lemma}

\begin{proof}
 Apply \Cref{r1} and \Cref{adjointoffourier}.
\end{proof}

\subsubsection{Behaviour under inclusions, Fourier transforms and rotation maps}
\( \)

This short subsection discusses the relationships between
biprojections, biunitaries and bipartial isometries. Although we do not
use any result of this section, we include it to show how
bipartial isometry comes  naturally  into the picture to describe
higher dimensional biprojections, which generalizes Bisch's
biprojections which are elements of the ``two-box space''.
\smallskip

An arbitrarily fixed $u\in \text{BU}_k(B,A)$ may also be thought of
as a unitary element in $B^{\prime}\cap A_{k+1}$ via the canonical
inclusion map. A natural question is whether $u\in
\text{BU}_{k+1}(B,A)$ or not? This might not be the case
always. However, $u\in \text{BPI}_{k+1}(B,A)$ as the following result
shows.
\begin{proposition} 
  \begin{enumerate}
    \item If $u\in \mathrm{BU}_k(B,A)$, then $u\in
      \mathrm{BPI}_{k+1}(B,A)$ for all $k \geq 1$.
\item If $e\in \mathrm{BP}_1(B,A)$, then $e\in \mathrm{BPI}_{k+1}(B,A)$ for all $k \geq 1$.
\end{enumerate}  \end{proposition}
\begin{proof}
 (1): We just need to prove that $\mathcal{F}_{k+1}(u)$
 is a multiple of a partial isometry. Indeed, since
 $ue_{k+2}=e_{k+2}u$ it readily follows that
 $$\mathcal{F}_{k+1}(u)={\tau}^{-1} e_{k+2}\mathcal{F}_k(u).$$ Hence,
 ${\mathcal{F}_{k+1}(u)}^*\mathcal{F}_{k+1}(u)={\tau}^{-2}{\mathcal{F}_{k}(u)}^*e_{k+2}\mathcal{F}_{k}(u)$. Since
 $\mathcal{F}_k(u)$ is a unitary this proves that
 $\mathcal{F}_{k+1}(u)$ is a multiple of a partial isometry. This
 finishes the proof.
\smallskip

(2): We need to show that $\mathcal{F}_k (e)$ is a multiple of a
partial isometry for all $k \geq 1$. For $k=1$ it follows from
\Cref{bipro} that $\mathcal{F}_k(e)=tf$ for some projection $f$. Thus
we need to prove only for $k\geq 2.$ To see this observe that
\begin{align*}
  \mathcal{F}_k(e)
  &= {\tau}^{-\frac{k+2}{2}} E^{B^{\prime}\cap A_{k+1}}_{A^{\prime}\cap A_{k+1}}\big(ee_{k+1}e_k\cdots e_2e_1\big)\\
  &= {\tau}^{-\frac{k+2}{2}}E^{B^{\prime}\cap A_{k+1}}_{A^{\prime}\cap A_{k+1}}\big(e_{k+1}e_k\cdots e_3ee_2e_1\big)\\
  &={\tau}^{-\frac{k+2}{2}}e_{k+1}e_k\cdots e_3E^{B^{\prime}\cap A_{k+1}}_{A^{\prime}\cap A_{k+1}}\big(ee_2e_1\big)\\
  &= {\tau}^{-\frac{k-1}{2}}te_{k+1}e_k\cdots e_3f.
  \end{align*}
Thus,
$$
{\big(\mathcal{F}_k(e)\big)}^*\mathcal{F}_k(e)=t{\tau}^{-\frac{1}{2}}
\frac{fe_3f}{\tr(f)}.
$$
And, we saw in the proof of \Cref{efisabiprojection} that
$\frac{fe_3f}{\tr(f)}$ is a projection. Thus, $\mathcal{F}_k(e)$ is a
multiple of a partial isometry. This completes the proof.
\end{proof}

The next  result shows that Fourier transform of a biprojection
(in $B^{\prime}\cap A_1$) is a bipartial isometry.

\begin{proposition}
If $e\in \mathrm{BP}_1(B,A)$, then $\mathcal{F}_1(e) \in \mathrm{BPI}_2(B,A).$
\end{proposition}

\begin{proof}
 Write, as before, $\mathcal{F}_1(e)=tf$ for some scalar $t > 0$ and
some  projection $f\in A'\cap A_2\subset B'\cap A_2$. Then,
 \[
  \mathcal{F}_2(f)
  = {\tau}^{-2} E^{B^{\prime}\cap A_3}_{A^{\prime}\cap A_3}\big(fe_3e_2e_1\big)
  ={\tau}^{-1} fe_3e_2.
  \]
  Now, as before, $\frac{fe_3f}{tr(f)}$ being a projection, it is
   obvious that $fe_3e_2$, and hence $\mathcal{F}_2(f)$, is a
  multiple of a partial isometry.
\end{proof}

The following result shows that the biunitaries and biprojections (in
$B^{\prime}\cap A_1$) are preserved under the rotation map.
\begin{proposition}
  \begin{enumerate}
    \item If $u\in \mathrm{BU}_1(B,A) $, then 
      $\gamma_0(u), \gamma_0(u^*)\in \mathrm{BU}_1(B,A) $.
 \item  If  $e\in \mathrm{BP}_1(B,A)$, then $\gamma_0(e)\in \mathrm{BP}_1(B,A)$.
\end{enumerate}
  \end{proposition}

\begin{proof}
  (1) follows from \Cref{adjointoffourier} and \Cref{r1}.\smallskip

  (2): By \Cref{r1}, it follows immediately that
${\gamma_0(e)}^*=\gamma_0(e).$  Also,
\begin{align*}
 {\gamma_0(e) }^2
 &= {\tau}^{-2}\sum_{i,j} E_1(e_1\lambda_ie)e_1\lambda^*_iE_1(e_1\lambda_je)e_1\lambda^*_j\\
 &= {\tau}^{-2}\sum_{i,j}E_1(e_1\lambda_ie)E_0\big(\lambda^*_iE_1(e_1\lambda_je)\big)e_1\lambda^*_j\\
 &= {\tau}^{-2}\sum_{i,j} E_1\bigg(e_1\lambda_iE_0\big(\lambda^*_iE_1(e_1\lambda_je)\big)e\bigg)e_1\lambda^*_j\\
 &= {\tau}^{-2}\sum_j E_1\bigg(e_1E_1(e_1\lambda_je)e\bigg)e_1\lambda^*_j\\
 &={\tau}^{-1}\sum_j E_1(e_1\lambda_je)e_1\lambda^*_j= \gamma_0(e).
\end{align*}
Thus, $\gamma_0(e)$ is a projection. On the other hand, by
\Cref{adjointoffourier}, we have
$$\mathcal{F}_1\big(\gamma_0(e)\big)={\big(\mathcal{F}_1(e)\big)}^*=tf.$$
Hence, $\gamma_0(e)$ is a biprojection. 
\end{proof}

In this paper, we discuss only about biprojections. The applications
  of biunitaries will be analyzed in a future paper. 

\subsection{From biprojections to intermediate $C^*$-subalgebras}\( \)

As in the case of a type $II_1$-subfactor, we have:
\begin{lemma}  \label{bisc1} 
  \(
  B=\{e_1\}^{\prime}\cap A.
  \)
\end{lemma}

\begin{proof}
 Clearly, $B\subseteq \{e_1\}^{\prime}\cap A.$ To see the other
 direction, take $x\in \{e_1\}^{\prime}\cap A.$ Then, 
 $e_1x=e_1xe_1=e_1E_0(x).$ Then, by Lemma \ref{pushdown}, we see that
 $x=E_0(x)\in B.$
\end{proof}

An abstract characterization of the intermediate subfactors of a type
$II_1$ subfactor $N \subset M$ in terms of biprojections was
established by Bisch in \cite[Theorem 3.2]{Bi2}. His proof crucially
uses a canonical conjugate linear unitary operator on the standard
space $L^2(M, \tr_M)$ and the notion of downward basic construction of
a subfactor, both of which are not available for inclusions of general
simple $C^*$-algebras.  Still, based on the $C^*$-Fourier theory
developed above, a suitable adaptation of Bisch's proof yields the
following recipe to obtain intermediate $C^*$-subalgebras of the dual pair. Recall that, given any intermediate $C^*$-algebra $C$ of $B\subset  A$, the Jones projection $e_C$ will always be a biprojection. In this sense, the following can be thought of as a partial converse.

\begin{theorem} \label{bischprojection} \label{bibiprojection2}
Let $B\subset A$ be an irreducible inclusion of simple
  unital $C^*$-algebras with a conditional expectation $E$ of finite
  index and $e\in B^{\prime}\cap A_1$ be a biprojection. Then,
  $[\mathcal{F}(e)]$ implements a conditional expectation onto the
  intermediate $C^*$-subalgebra $P$ of $A \subset A_1$ given by $P=
  [\mathcal{F}(e)]'\cap A_1$. Moreover,
 \[
 P = AeA, \ 
 [\mathcal{F}(e)] = e_{P}  \text{ and } {[A_1:P]}_0={\tr\big( [\mathcal{F}(e)]}\big)^{-1}.
\]
\end{theorem}

\begin{proof}
  First, recall that, since $B \subset A$ is irreducible, $E$ is the
  unique minimal conditional expectation from $A$ onto $B$ (see
  \Cref{wata1}).  Now, let $f$ denote the projection
  $[\mathcal{F}(e)]$ in $A'\cap A_2$ and $P:= \{f\}'\cap
  A_1$. \smallskip

  \noindent{\bf Step I:} For each $x_1\in A_1$, we show that
  $fx_1f=y_1 f$ for some $y_1\in A_1$.

  Since $ A_1 = A e_1 A$, it suffices to consider $x_1$ of the form
  $x_1=ae_1b$ for some $a,b\in A$. Since $f \in A'\cap A_2$, we get
  $fx_1f=fae_1bf=afe_1f b.$ But,
\begin{equation}\label{importantequation2}
fe_1f 
 ={\tau}^{-1} fe_1e_2e_1f
=t^{-2} ee_2e,
     \end{equation}
where the last equality follows from Lemma \ref{fe1e2}. Thus, applying
Eq.\eqref{ef=fe} followed by \Cref{ee1}, we obtain
\[
fx_1 f = t^{-2} a (e e_2 e) b  = t^{-2} \tr(e)
a(ef)b=t^{-2} \tr(e) aebf= \tr(f) aebf.
\]
 Putting $\tr(f) aeb=y_1\in A_1$, we get $fx_1f=y_1f.$

\smallskip

\noindent{\bf Step II:}  We show that $y_1$ obtained in Step I is an element
of $P.$

Without loss of generality we may assume that $x_1=x^*_1.$ Then,
$y_1f=fx_1f={(fx_1f)}^*=fy^*_1$; so that $E_2(y_1f)=E_2(fy^*_1)$,
which implies that $y_1=y^*_1$, by Lemma \ref{scalar1}(2). In particular,
$fy_1=y_1 f$ and, hence, $y_1\in \{f\}^{\prime}\cap A_1=P.$
\smallskip

\noindent{\bf Step III:} We show that $f$ implements a conditional
expectation from $A_1$ onto $P$. Since $B^{\prime}\cap A=\C$ it follows that $P$ is a simple
  $C^*$-subalgebra of $A\subset A_1$. Let us denote by $E_{P}$ the unique minimal conditional expectation
from $A_1$ onto $P$.  We shall show that
$fx_1f=E_{P}(x_1)f$. Thanks to above steps and Lemma \ref{scalar1}(2),
it is clear that for any $x_1\in A_1$ there exists a unique $y_1\in
P$ such that $fx_1f=y_1f.$  Pefine $\phi:A_1\rightarrow P$ by
$\phi(x_1)=y_1$. We show that $\phi=E_{P}.$

 To see that $\phi$ is {positive}, assume that $x_1$ is
 positive, i.e. $x_1=s^*s$ for some element $s = \sum_{i=1}^n a_i e_1
 b_i\in A_1$ for some $a_i, b_i \in A$. Then, as in Step I, we get
 $y_1 = \sum_{i,j} {\tau}^{1/2} b^*_i e E_0(a^*_i a_j)eb_j$. Since a
 conditional expectation is always completely positive, we have $y_1
 \geq 0$.
  
 Next, we show that $\phi$ is a {retraction} on
 $P$. This is obvious. If
 $x_1\in P$, then by the definition of $P$, we must have
 $fx_1f=x_1f$; so, by the uniqueness of $y_1$, we immediately get
 $y_1=x_1\in P$. In other words, $\phi(x_1)=x_1$ for all $x_1\in
 P.$ In particular, $\|\phi\| = \|\phi(1_{A_1})\| = \|1_{P}\| =1$.
 
 Finally, we show that $\phi$ is a $P$-bimodule map. This also follows
 trivially. Let $c_1, c_2\in P$.    By the definition of $P$, we have $f\in {P}^{\prime}\cap A_2$; so, for each $x_1\in A_1$, we see that 
 \[
  \phi(c_1x_1c_2) f
  = fc_1x_1c_2f
  = c_1fx_1fc_2
  = c_1\phi(x_1)fc_2= c_1\phi(x_1)c_2f.
 \]
Hence, by the definition of $\phi$, we obtain $\phi(c_1x_1c_2) = c_1
\phi(x_1) c_2$.
 
Thus, $\phi$ implements a conditional expectation from $A_1$ onto
$P$. Since $B \subset A$ is irreducible, we must have
$\phi=E_{P}$, by \Cref{wata1}. Hence, $fx_1f=E_{P}(x_1)f$ as
desired.\\

In order to show that $P=AeA$, first note that, \(
AeAf=AefA=A(ee_2e)A\). Then, by Eq. \eqref{importantequation2}, we
obtain $$A(ee_2e)A=A(fe_1f)A=fA_1f=E_{P}(A_1)f=Pf.$$ Thus,
$(AeA)f=Pf$ and hence $AeA=P.$

Finally, we show that $[\mathcal{F}(e)] = e_{P}$. 

 First note that,
 \begin{align*}\vspace*{-2mm}
  {\mathcal{F}_1}^{-1}(e_{P}) 
  &= {\tau}^{-3/2} E_2(e_{P}e_1e_2)\\
  &= {\tau}^{-3/2} E_2(e_{P}e_1e_{P}e_2)\\
  &= {\tau}^{-3/2} E_2\big(E^{A_1}_{P}(e_1)e_2\big)\\
  &= {\tau}^{-1/2} E^{A_1}_{P}(e_1).
 \end{align*}
On the other hand,\vspace*{-2mm}
\begin{align*}
 {\mathcal{F}}_1\big(E^{A_1}_{P}(e_1)\big) 
 &= {\tau}^{-3/2} E^{B^{\prime}\cap A_2}_{A^{\prime}\cap A_2}\big(E^{A_1}_{P}(e_1)e_2e_1\big)\\
  &= {\tau}^{-3/2} E^{B^{\prime}\cap A_2}_{A^{\prime}\cap A_2}\big(E^{A_1}_{P}(e_1)fe_2e_1\big)\\
  &= {\tau}^{-3/2} E^{B^{\prime}\cap A_2}_{A^{\prime}\cap A_2}(fe_1fe_2e_1\big)\\
  &= {\tau}^{-1/2}  E^{B^{\prime}\cap A_2}_{A^{\prime}\cap A_2}(fe_1)\\
  &= {\tau}^{1/2}f.
\end{align*}
Hence, $f=e_{P}$. This completes the proof.
\end{proof}
\begin{remark}

  We do not know at present whether the biprojection $`e$' corresponds
  to an intermediate $C^*$-subalgebra of $B \subset A$ or not. A
  natural candidate would be $C := \{e\}^{\prime}\cap A$. But we can't see how $e$ will implement the conditional expectation from $A$ onto $C$.

  Further we would
  like to know whether $P$ is the basic construction of the pair
  $C \subset A$ or not?  We feel that this may not
  be plausible because of possible $K$-theoretical obstruction. It
  will be interesting to analyze this in detail.
\end{remark}

\noindent Applying Remark \ref{bibiprojection1} and \Cref{bibiprojection2} we get the following.
\begin{corollary}
 If $e\in \mathrm{BP}_1(B,A)$ then $f \big(:=[\mathcal{F}(e)]\big)\in \mathrm{BP}_1(A,A_1)$, where $e$ is as in \Cref{bischprojection}.
\end{corollary}

\section{An angle between intermediate $C^*$-subalgebras}

In this section, motivated by \cite{BDLR}, we introduce the notions of
interior and exterior angles between any two intermediate
$C^*$-subalgebras $C$ and $D$ of a given inclusion $B \subset A$ of
unital $C^*$-algebras. As mentioned in the Introduction, a significant
application of the notion of angle in \cite{BDLR} was to better
Longo's bound for the cardinality of the lattice of intermediate
subfactors of type $II_1$, thereby answering a question of Longo. On
similar lines, towards the end of this section, we exploit our notion
of interior angle and some aspects of the $C^*$-Fourier theory
that we developed to
\begin{enumerate}[(a)]
\item obtain a bound for the cardinality of the set of
intermediate $C^*$-subalgebras of an irreducible inclusion of simple
unital $C^*$-algebras, and
\item improve Longo's upper bound for the
cardinality of intermediate subfactors of an irreducible subfactor of
type $III$.
\end{enumerate}

\subsection{Interior and exterior angles between intermediate $C^*$-subalgebras}\( \)

 Let $B\subset A$ be an inclusion of unital $C^*$-algebras and a conditional expectation $E: A \rar B$ of
finite index.  By Lemma \ref{Aisahilbertmodule}, $A_1$ is a Hilbert
$A$-module with respect to the $A$-valued inner product ${\langle a_1,
  b_1\rangle}_A:=\widetilde{E}(a^*_1b_1)$ for $a_1,b_1\in A_1.$

As in \cite{IW}, let $\mathrm{IMS}(B, A, E)$ denote the set all
intermediate $C^*$-subalgebras $C$ of $B \subset A$ with a conditional
expectation $F: A \rar C$ such that $ E_{|_C} \circ F = E$. Then, as in \Cref{A-C-B},
it is easily  seen that $C_1 \subset A_1$.  For any pair
$C,D\in \text{IMS}(B,A,E)$, let $e_C$ and $ e_D$ denote the
corresponding Jones projections in $C_1$ and $D_1$,
respectively. Then, by Cauchy-Schwarz inequality (see \cite{P}), we
have
\begin{equation}
\lVert{\langle
    e_C-e_B,e_D-e_B\rangle}_{A}\rVert \leq \lVert
    e_C-e_B\rVert_{A}  \lVert
    e_D-e_B\rVert_{A}.
\end{equation}
Based on this, we propose the following:
\begin{definition}\label{angles}
  Let  $(B,C,D,A)$ be a quadruple of $C^*$-algebras as above.
  \begin{enumerate}\item
 The  \textit{interior angle} between $C$ and $D$, denoted by
$\alpha^B_A(C,D)$, is defined by the expression
\[
\cos\Big(\alpha^B_A(C,D)\Big)=\displaystyle \frac{\lVert{\langle
    e_C-e_B,e_D-e_B\rangle}_{A}\rVert}{{\lVert
    e_C-e_B\rVert}_{A}{\lVert
    e_D-e_B\rVert}_{A}}.
\]
\item The \textit{exterior angle} (or
 \textit{dual angle}) between $C$ and $D$, denoted by  $\beta^A_B(C,D)$, is defined
 as the interior angle between $C_1$ and $D_1$, that is,
 \[
\beta^A_B(C,D):= \alpha^A_{A_1}(C_1,D_1).
\]
\end{enumerate}
We take  the value of $\alpha$ (and, hence, of $\beta$) only in the interval $[0, \pi/2]$. \end{definition}

\begin{remark}  \label{angleforsimple}
  \begin{enumerate}
\item It is clear that the definition of $\alpha(P, Q)$ (and
  $\beta(P,Q)$) depends on the conditional expectations from $A$ onto
  $C$ and $D$. However, when $B \subset A$ is an irreducible inclusion
  of simple unital $C^*$-algebras,
  then by \Cref{wata1}, there will be unique (minimal) conditional
  expectations from $A$ onto $C$ and $D$ and hence there won't be any
  ambiguity.
  \item If $B \subset A$ is a pair of simple unital $C^*$-algebras and
    $C$ and $D$ are both simple, then one can work with the minimal
    conditional expectations.
  \item If $B \subset A$ is an irreducible inclusion of simple
      unital $C^*$-algebras such that $\mathcal{E}_0(A, B) \neq
      \emptyset$ and $E^A_B: A \rar B$ is the unique minimal
      conditional expectation. Then, by \Cref{intermediate-ce},
      $\mathrm{IMS}(B,A,E^A_B)$ consists of all intermediate
      $C^*$-subalgebras of $B \subset A$.
\item For a subfactor of type $II_1$ factor with finite
  Jones index, the interior angle defined here is different from that
  in \cite{BDLR}. Recall, the trace preserving conditional expectation need not be the minimal conditional expectation.
\end{enumerate}
  \end{remark}

By a quadruple $\mathcal{G}=(B,C,D,A)$ of unital
$C^*$-algebras, we shall mean that $A$ is a unital $C^*$-algebra 
with  unital $C^*$-subalgebras  $B, C$ and $D$ such that $B
\subseteq C \cap D$. The quadruple $\mathcal{G}$ is said to be
irreducible if $B \subset A$ is an irreducible inclusion.\smallskip

We will be interested only in analyzing intermediate $C^*$-subalgebras
of simple inclusions.

\subsubsection{Angle between intermediate $C^*$-subalgebras of a simple unital inclusion}\( \) 

Throughout this subsection, $\mathcal{G}=(B,C,D,A)$ will denote a
quadruple of simple unital $C^*$-algebras such that
$\mathcal{E}_0(A,B)\neq \emptyset.$ We first list some notations that
will be used ahead.
\begin{notation}\label{p}
     \begin{enumerate}
\item We  denote the corresponding unique minimal conditional expectations by
  $E^A_C,E^A_D$ and $E^A_B$ (see \Cref{min1}).  
 \item We let $e_C$ and $e_D$ denote the Jones projections
   corresponding to $E^A_C$ and
   $E^A_D$, respectively.
\item As in \Cref{A-C-B}, $C_1$ and $D_1$ are both simple and contained
  in $A_1$; the quadruple $\widetilde{\mathcal{G}}:=(A,C_1,D_1,A_1)$
  will be called the dual quadruple.
\item Using multiplicativity of
the Watatani index (\Cref{min3}), it is clear that $
\frac{{[C:B]}_0}{{[A:D]}_0}= \frac{{[D:B]}_0}{{[A:C]}_0}.$ We denote this
common value by $r$. If $r=1$ we call the quadruple a parallelogram.
\item The quadruple $\mathcal{G}$ will be called  a commuting square if $E^A_C E^A_D= E^A_D
E^A_C = E^A_B$. And $\mathcal{G}$ will be  called a co-commuting square if
the dual quadruple $\widetilde{\mathcal{G}}$ is a commuting square.
     \end{enumerate}
\end{notation}

\begin{remark}
 As in \cite{BDLR}, it can be seen easily that $\alpha(C,D)=\pi/2$ if and
 only if $(B,C,D,A)$ is a commuting square. The dual statement holds
 similarly; thus, $\beta(C,D)=\pi/2$ if and only if $(B,C,D,A)$ is a
 co-commuting square.

   We may also do a $C^*$-algebraic version of the theory developed in
   Section 2 of \cite{BDLR}. The details are similar and left to the
   interested readers.
\end{remark}

We now provide some useful expressions for  the interior and
exterior angles.
\begin{proposition} \label{relationshipbetweenalphanadbeta}
We have
  \[
  \cos\big(\alpha^B_A (C,D)\big)=r\displaystyle
 \frac{\sqrt{{[A:C]}_0-1}\sqrt{{[A:D]}_0-1}}{\sqrt{{[C:B]}_0-1}\sqrt{{[D:B]}_0-1}}\cos\big(\beta^B_A(C,D)\big)+
 \frac{r-1}{\sqrt{{[C:B]}_0-1}\sqrt{{[D:B]}_0-1}}.
 \]
\end{proposition}
\begin{proof}
 Let $v_C:=\frac{e_C-e_1}{{\lVert e_C-e_1\rVert}_2},$ where $\|\cdot\|_2$ is defined with respect to the tracial state $\tr$ on $B'\cap A_1$ as in \Cref{f1}. Then, applying \Cref{fouriertransformofec1}, we get
 \[
 \mathcal{F}(v_C)=\displaystyle
 \frac{\frac{\sqrt{{[A:B]}_0}}{{[A:C]}_0}e_{C_1}-\frac{1}{\sqrt{{[A:B]}_0}}1_{A_1}}{{\lVert
     e_C-e_1\rVert}_2}.
 \]
By definition, $\cos\big(\alpha^B_A(C,D)\big)=\langle v_C,v_D\rangle.$ So, by \Cref{F}, we obtain
\begin{eqnarray*}
\lefteqn{ \cos\big(\alpha^B_A(C,D)\big)}\\ & = & \big\langle
        {\mathcal{F}}(v_C),\mathcal{F}(v_D)\big \rangle\\ &=
        &\frac{1}{{\lVert e_C-e_1\rVert}_2{\lVert
            e_D-e_1\rVert}_2}\Bigg\langle
        \frac{{[A:B]}_0}{{[A:C]}_0}e_{C_1}-\frac{1}{\sqrt{{[A:B]}_0}}
        1_{A_1},
        \frac{{[A:B]}_0}{{[A:D]}_0}e_{D_1}-\frac{1}{\sqrt{{[A:B]}_0}}
        1_{A_1} \Bigg\rangle\\ &= & \frac{1}{{\lVert e_C-e_1\rVert}_2{\lVert
            e_D-e_1\rVert}_2}\Bigg(r\langle
        e_{C_1},e_{D_1}\rangle-\frac{1}{{[A:D]}_0}
        tr(e_{D_1})-\frac{1}{{[A:C]}_0}
        tr(e_{C_1})+\frac{1}{{[A:B]}_0}\Bigg)\\ &= & \frac{1}{{\lVert
            e_C-e_1\rVert}_2{\lVert e_D-e_1\rVert}_2}\Bigg(r
        \cos\big(\beta^B_A(C,D)\big) {\lVert
          e_{C_1}-e_2\rVert}_2{\lVert
          e_{D_1}-e_2\rVert}_2-\frac{1}{{[A:B]}_0}+\frac{r}{{[A:B]}_0}\Bigg)\\
        & = & r\displaystyle
 \frac{\sqrt{{[A:C]}_0-1}\sqrt{{[A:D]}_0-1}}{\sqrt{{[C:B]}_0-1}\sqrt{{[D:B]}_0-1}}\cos\big(\beta^B_A(C,D)\big)+
 \frac{r-1}{\sqrt{{[C:B]}_0-1}\sqrt{{[D:B]}_0-1}}.
        \end{eqnarray*}
\end{proof}

 From the preceding proposition, the following three
results follow easily. We leave it to the reader to check the
details. These results have been mentioned for $II_1$ factors in
\cite{BDLR} and \cite{BG2}.
 
 The following result says that if a commuting square
 (resp. co-commuting square) is a parallelogram then it must be a
 co-commuting square (resp. commuting square).

 \begin{corollary}
 If $\beta^B_A(C,D)=\pi/2$, then
 $$\cos\big(\alpha^B_A(C,D)\big)= \displaystyle \frac{r-1}{\sqrt{{[C:B]}_0-1}\sqrt{{[D:B]}_0-1}}.$$
 On the other hand, if $\alpha=\pi/2$, then
 $$ \cos \big(\beta^B_A(C,D)\big)=\displaystyle \frac{\frac{1}{r}-1}{\sqrt{{[A:C]}_0-1}\sqrt{{[A:D]}_0-1}}.$$
\end{corollary}

 In particular, the preceding result says that if $(B,C,D,A)$ is a
commuting square then ${[A:D]}_0\geq {[C:B]}_0$ (and so,
${[A:C]}_0\geq {[D:B]}_0) $. On the other hand for a co-commuting
square ${[A:D]}_0\leq {[C:B]}_0$ (and so, ${[A:C]}_0\leq {[D:B]}_0$.
\begin{corollary}
 If $(B,C,D,A)$ is a parallelogram (that is, if $r=1$) then
 $\alpha^B_A(C,D)= \beta^B_A(C,D).$
\end{corollary}
In view of the above corollary we can interpret the interior and
exterior angles as opposite angles of a parallelogram.
By definition, $\cos \big(\beta^{A}_{A_1} (C_1,D_1)\big) =\cos \big(\alpha^{A_1}_{A_2} (C_2,D_2)\big).$ However, we can say the following.
\begin{corollary}\label{D}
 $\cos \big(\alpha^B_A(C,D)\big)=\cos \big(\beta^{A}_{A_1} (C_1,D_1)\big).$
 \end{corollary}
\begin{proof}
Apply \Cref{relationshipbetweenalphanadbeta} for the quadruple
$(A,C_1,D_1,A_1)$ to get
$$ \cos\big(\beta^B_A (C,D)\big)=\frac{1}{r}\displaystyle
\frac{\sqrt{{[C:B]}_0-1}\sqrt{{[D:B]}_0-1}}{\sqrt{{[A:C]}_0-1}\sqrt{{[A:D]}_0-1}}\cos\big(\beta^{A}_{A_1}(C_1,D_1)\big)+
\frac{\frac{1}{r}-1}{\sqrt{{[A:C]}_0-1}\sqrt{{[A:D]}_0-1}}.$$ Again
apply \Cref{relationshipbetweenalphanadbeta} to obtain the desired
result. Details are obvious.
 \end{proof}
The above result justifies the name `dual angle'.
\begin{remark}
 \Cref{D} can also be independently proved using \Cref{shift}.
\end{remark}

A priori, it not clear why $\alpha(C, D) = 0 $ if and
only if $C = D$. However, when the pair $B \subset A$ is irreducible,
we have:

\begin{proposition}
  Let $(B, C, D, A)$ be an irreducible quadruple of simple unital
  $C^*$-algebras such that $\mathcal{E}_0(A, B) \neq \emptyset$.  Then,
  \begin{enumerate}
  \item $\alpha(C,D)=0$ if and only if $C=D$; and
    \item the
  interior and exterior angles between $P$ and $Q$ are given by
\begin{equation} \label{traceformulaofangle}
 \cos \big(\alpha^B_A(C,D)\big)=\displaystyle \frac{ \tr (e_Ce_D)-\tau}{\sqrt{\tr (e_C) -\tau}\sqrt{\tr (e_D) -\tau}},\ \text{and}
\end{equation}
\begin{equation}
\cos\big(\beta^B_A(C,D)\big)=\displaystyle \frac{\tr(e_C e_D)-\tr (e_C)\, \tr (e_D)}{\sqrt{\tr (e_C)- {\tr (e_C)}^2}\sqrt{\tr (e_D)- {\tr (e_D)}^2}}.
\end{equation}
\end{enumerate}
\end{proposition}

\begin{proof}
(1): By \Cref{rel-comm}, $B^{\prime}\cap A_1$ is finite dimensional; so,
  it is a Hilbert space with respect to the inner product induced by
  the tracial state $\tr$. Further, if $a_1,b_1\in B^{\prime}\cap
  A_1$, then $E_1(a^*_1b_1)\in B^{\prime}\cap A=\C$ and, therefore,
  ${\langle a_1,b_1\rangle}_A= E_1(a^*_1b_1)=E_0\circ
  E_1(a^*_1b_1)=\tr(a^*_1b_1)=\langle a_1,b_1\rangle.$ One then easily
  deduces that $\alpha(C,D)=0$ if and only if $C=D.$ See, for
  instance, \cite[Proposition 2.3]{BDLR}.\smallskip

(2): The formulas for angles follow easily (as in \cite{BDLR}) and we
  omit the details. For example, formula for $\alpha$ follows from the
  fact that $E_A(e_Ce_D-e_B)\in B^{\prime}\cap A=\C$.
\end{proof}

We now proceed to apply our  notions of $C^*$-Fourier theory and interior
angle to obtain a bound on the cardinality of intermediate
$C^*$-subalgebras of a given inclusion of simple $C^*$-algebras. Note
that if the pair $B \subset A$ is not irreducible, then its
intermediate $C^*$-algebras are not necessarily simple and conjugating
by unitaries of $B'\cap A$, we obtain infinitely many intermediate
$C^*$-subalgebras from any given one. So, we will now restrict our
analysis to irreducible pairs only.\smallskip

First, we develop some auxiliary results (which are also of
independent interest) that will be required ahead.

\subsection{Two auxiliary operators associated to an irreducible quadruple of $C^*$-algebras}\( \)

Throughout this subsection, $\mathcal{G}=(B,C,D,A)$
will be assumed to be an irreducible quadruple of simple unital
$C^*$-algebras such that $\mathcal{E}_0(A, B) \neq \emptyset$.

Fix any two quasi-bases $\{\gamma_j:1\leq j\leq m\}$ and $\{\delta_k:
1\leq k\leq l\}$ for $E^C_B$ and $E^D_B$, respectively. We wish to
show that the two positive elements
\[
\sum_{j,k}
\gamma_j\delta_ke_B\delta^*_k\gamma^*_j \text{ and }
\sum_{j,k} \delta_k\gamma_j e_B \gamma^*_j\delta^*_k
\]
remain same even when we vary the quasi-bases. Similar operators have
also been  studied and used in \cite{BDLR, BG2, BG}.

\begin{lemma}  \label{intermediatebasis}
 We have
 \[
 \sum_k\delta_k e_B\delta^*_k=e_D\text{ and }
 \sum_j \gamma_j e_B \gamma_j^*=e_C.
 \]
In particular, $\sum_{j,k} \gamma_j\delta_ke_B\delta^*_k\gamma^*_j
 \in D_1$ and $\sum_{j,k} \delta_k\gamma_j e_B \gamma^*_j\delta^*_k
 \in C_1$. 
\end{lemma}
\begin{proof}
 For any $a\in A$, we have
 \begin{align*}
  \sum_k\delta_ke_B\delta^*_k(a)= & \sum_k\delta_kE^A_B(\delta^*_ka)\\
  =& \sum_k \delta_kE^D_B\circ E^A_D(\delta^*_ka)\\
  =& \sum_k \delta_kE^D_B\Big(\delta^*_k E^A_D(a)\Big)\\
  =& E^A_D(a) \text{\hspace*{10mm} (since}~~~\{\delta_k\}~~~\text{is a quasi-basis for}~~E^D_B)\\
  =& e_D(a).
 \end{align*}
Similarly, for $e_C$.
\end{proof}

\begin{proposition} \label{formula of p} We have
  \begin{enumerate}
    \item $E^{B^{\prime}\cap A_1}_{C^{\prime}\cap A_1}(e_D)
      ={[C:B]}_0^{-1} \sum_{j,k}
      \gamma_j\delta_ke_B\delta^*_k\gamma^*_j $ and
      \item 
        $ E^{B^{\prime}\cap A_1}_{D^{\prime}\cap A_1}(e_C) =
        {[D:B]}_0^{-1} \sum_{j,k} \delta_k\gamma_j e_B \gamma^*_j\delta^*_k.$
        \end{enumerate}
  In particular, \( \sum_{j,k}
      \gamma_j\delta_ke_B\delta^*_k\gamma^*_j
      \in C'\cap D_1$, $
 \sum_{j,k} \delta_k\gamma_j e_B \gamma^*_j\delta^*_k
 \in D'\cap C_1\) and they are independent of
 the quasi-bases $\{\gamma_j\}$ and $\{\delta_k\}$.
\end{proposition}

\begin{proof}
(1): In view of \Cref{intermediatebasis}, it suffices to generalize Lemma
\ref{f2} as follows:

We show that
 \begin{equation}\label{importantequation3}
 E^{B^{\prime}\cap A_k}_{C^{\prime}\cap A_k}(x)={{[C:B]}_0}^{-1}\sum_j \gamma_jx\gamma^*_j\ \text{for all } x\in B^{\prime}\cap A_k.
 \end{equation}
 It is easy to see that $\sum_j \gamma_j x\gamma^*_j\in C^{\prime}\cap A_k$. Indeed, for any $c\in C$, we have 
  \begin{align*}
  \sum_j \gamma_jx\gamma^*_j c
  &= \sum_{j,j'} \gamma_jxE^C_B(\gamma^*_jc\gamma_{j'})\gamma^*_{j'}\\
  &= \sum_{j,j'} \gamma_j E^C_B(\gamma^*_jc\gamma_{j'})x\gamma^*_{j'}~~~~\hspace*{20mm} (\textrm{since}~~x\in B^{\prime})\\
  &= c\sum_j \gamma_{j'}x\gamma^*_{j'}.
  \end{align*}
So, it now suffices to show that
$\tr\big(\sum_j\gamma_jx\gamma^*_jy\big)={[B:C]}_0\tr(xy)$ for all
$y\in C^{\prime}\cap A_k$. For any such  $y$, we have
\begin{align*}
  \tr\big(\sum_j\gamma_jx\gamma^*_jy\big) &= E_0\circ E_1\circ\cdots
  E_k\big(\sum_j\gamma_jx\gamma^*_jy\big)\\ &= E_0\circ E_1\circ\cdots
  E_k \big(\sum_j \gamma_j
  xy\gamma^*_j\big) \hspace*{30mm}  (\textrm{since}~~y\in C^{\prime})\\ &=
  E_0\bigg(\sum_j\gamma_j\big(E_1\circ \cdots \circ
  E_k(xy)\big)\gamma^*_j\bigg)\\ &= {[B:C]}_0 \tr(xy),
  \end{align*}
where the last equality follows because $E_1\circ\cdots \circ E_k(xy)
\in B^{\prime}\cap A=\C$ and is thus equal to $\tr(xy)$. Thus,
Eq.\eqref{importantequation3} holds.\smallskip

(2) follows by symmetry.
\end{proof}

\begin{definition}
  We define two positive elements $p(C,D)$ and $q(C,D)$ in $C'\cap D_1$ and $D'\cap C_1$, respectively, by
  \begin{equation}\label{p-and-q}
  p(C,D)=\sum_{i,j} \gamma_i\delta_je_B\delta^*_j\gamma^*_i
  \text{ and  } q(C,D)= \sum_{i,j} \delta_j\gamma_i e_B \gamma^*_i\delta^*_j
  \end{equation}
  for any two quasi-bases  $\{\gamma_i:1\leq i\leq m\}$ and $\{\delta_j: 1\leq j\leq n\}$ 
for $E^C_B$ and $E^D_B$, respectively,.
\end{definition}

Interestingly, these two auxiliary operators get mapped to each other under the
rotation on $B'\cap A_1$.  \color{black}
\begin{proposition}
 \label{relation between p and q}
 ${\gamma}_0\big(p(C,D)\big)=q(C,D)$ and ${\gamma}_0\big(q(C,D)\big)=p(C,D).$
\end{proposition}
\begin{proof}
As before, suppose $\{\lambda_i: 1 \leq i \leq n\}$, $\{\gamma_j:1\leq j \leq m \}$ and
$\{\delta_k:1\leq k\leq l \}$ are quasi-bases for $E^A_B$, $E^C_B$,
 $E^D_B$, respectively. We have
 \begin{align*}
  \gamma_0\big(p(C,D)\big)
  &= {\tau}^{-1} \sum_i E_1\bigg(e_1\lambda_i\big(\sum_j \gamma_je_D\gamma^*_j\big)\bigg)e_1\lambda^*_i\\
  &= {\tau}^{-1} \sum_{i,j} E_1\big(e_1e_D\lambda_i\gamma_je_D\gamma^*_j\big)e_1\lambda^*_i\\
  &= {\tau}^{-1}\sum_{i,j} E_1\bigg(e_1 E^A_D\big(\lambda_i\gamma_j\big)\gamma^*_j\bigg)e_1\lambda^*_i\\
  &= \sum_{i,j} E^A_D(\lambda_i\gamma_j)\gamma^*_je_1\lambda^*_i\\
  &= \sum_{i,j,k} \delta_k E^D_B\bigg(\delta^*_kE^A_D\big(\lambda_i\gamma_j\big)\bigg)\gamma^*_je_1\lambda^*_i\\
  &= \sum_{i,j,k} \delta_k E^D_B\circ E^A_D(\delta^*_k\lambda_i\gamma_j)\gamma^*_je_1\lambda^*_i\\
  &= \sum_{i,j,k} \delta_k E^C_B\circ E^A_C(\delta^*_k\lambda_i\gamma_j)\gamma^*_je_1\lambda^*_i\\
  &= \sum_{i,j,k} \delta_k  E^C_B\bigg(E^A_C\big(\delta^*_k\lambda_i\big)\gamma_j\bigg)\gamma^*_je_1\lambda^*_i\\
  &=\sum_{i,k} \delta_k E^A_C\big(\delta^*_k\lambda_i\big)e_1\lambda^*_i \hspace*{20mm} (\textrm{since}~~~\{\gamma_j\}~~~\textrm{is a quasi-basis for}~~~E^C_B)\\
  &= \sum_{i,k} \delta_k E^A_C\big(\delta^*_k\lambda_i\big)e_Ce_1\lambda^*_i\\
  &= \sum_{i,k} \delta_ke_C\delta^*_k\lambda_ie_1\lambda^*_i\\
  &=\sum_k \delta_ke_C\delta^*_k\\
  &=q(C,D).
 \end{align*}
 Then, on the other hand, applying Lemma \ref{r2}, we also obtain
 $\gamma_0\big(q(C,D)\big)=p(C,D).$
\end{proof}

We deduce some consequences that will be used ahead.
\begin{proposition} \label{pmultipleofprojection}
  There exists a positive scalar $t$ such that
  \begin{enumerate}
    \item $p(C,D)e_D=te_D$  and  $[p(C,D)]= \frac{1}{t} p(C,D)$;
\item $q(C, D)e_C = te_C$ and  $[q(C,D)] =
  \frac{1}{t} q(C,D)$; and
\item $p(C,D)e_C = t e_C$ and $q(C,D)e_D = te_D$,
  \end{enumerate}
  where $[x]$ denotes the support projection of $x$.\smallskip

  Moreover, \(e_C\vee e_D \leq [p(C, D)] \wedge [q(C, D)]\) and \( t =
  [A:B]_0 \tr(e_C e_D)\).  In particular, $e_C$ and $e_D$ are never
  orthogonal to each other.
\end{proposition}

\begin{proof}
(1): Since $p(C,D)\in D_1$, using Lemma \ref{pushdown}, we
have 
\[
p(C,D)e_D= [A:D]_0 E^{D_1}_A\Big(p(C,D)e_D\Big)e_D.
\]
Since $p(C,D)e_D\in B^{\prime}\cap D_1$, we must have
$E^{D_1}_A\big(p(C,D)e_D\big)\in B^{\prime}\cap A=\C.$ In other words,
$p(C,D)e_D=te_D$ for some scalar $t.$ 
Now, applying Proposition \ref{formula of p}, we obtain
 \begin{align*}
 {p(C,D)}^2= & {[C:B]}_0 p(C,D)E^{B^{\prime}\cap D_1}_{C^{\prime}\cap D_1}(e_D) \\
= & {[C:B]}_0 E^{B^{\prime}\cap D_1}_{C^{\prime}\cap D_1}\big(p(C,D)e_D\big)\\
=& t {[C:B]}_0 E^{B^{\prime}\cap D_1}_{C^{\prime}\cap D_1}(e_D)\\
=& t\, p(C,D).
\end{align*}
Since $ p(C,D)$ and $p(C,D)^2$ are both positive and
   non-zero, we must have $t > 0$. We also deduce that \( [p(C,D)]=
 \frac{1}{t} p(C,D)\geq e_D.\) \smallskip

(2): By symmetry, $q(C,D)$  also satisfies above properties.  To
show that the scalars have the same value, observe that, using
Theorem \ref{antiauto}, we obtain
\(
 {\gamma_0\big(p(C,D)\big)}^2=\gamma_0\big(p(C,D)^2\big).
\)
Then, by applying Proposition \ref{relation between p and q}, we get ${q(C,D)}^2=tq(C,D).$\smallskip

(3): We have $q(C,D)e_C=te_C$. Thus, by applying $\gamma_0$, we obtain
$te_C = e_C p(C,D) = p(C,D) e_C$, by \Cref{relation between p and q}
and \Cref{gamma0ec}.\smallskip

 From Item (2), we have \( [p(C,D)]= \frac{1}{t} p(C,D)\geq
e_D \) and \( [q(C,D)]= \frac{1}{t} q(C,D)\geq e_C\); and from Item (3), we
obtain $ [p(C,D)]\geq e_C$ and $ [q(C,D)] \geq e_D$.  Thus, $e_C \vee
e_D \leq [p(C,D)] \wedge [q(C,D)]$. \smallskip

Finally, from Item (3) and \Cref{intermediatebasis}, we obtain
$te_C=p(C,D)e_C=\sum_i\gamma_ie_De_C\gamma^*_i$; so that 
  \[
  \frac{t}{ [A:C]_0}=tE^{A_1}_A(e_C) = \sum_i\gamma_iE^{A_1}_A(e_De_C) \gamma_i^* = {[C:B]}_0E^{A_1}_A(e_De_C),
  \]
where the last equality follows from  the facts that
  $E^{A_1}_A(e_De_C)\in B^{\prime}\cap A=\C$ and that
  $\sum_i\gamma_i\gamma^*_i=[C:B]_0$. Hence, $t = [A:B]_0 \tr(e_C
e_D)$.
\end{proof}

We conclude this subsection with  some useful expressions for the above auxiliary operators.
\begin{proposition} \label{formulaofp2} We have
  \[
  p(C,D)={[D:B]}_0 E^{A_1}_{D_1}(e_C) \text{ and }
  q(C,D)={[C:B]}_0 E^{A_1}_{C_1}(e_D).
  \]
  In particular, $\tr(p(C,D))= r  = \tr(q(C,D))$.
\end{proposition} 

\begin{proof}
 To prove this we need the following general statement:
 \begin{equation}  \label{general statement}
\gamma_0\bigg(E^{B^{\prime}\cap A_1}_{D^{\prime}\cap A_1}(x)\bigg)=E^{A_1}_{D_1}\big(\gamma_0(x)\big)~~~~~\text{for any}~~~x\in B^{\prime}\cap A_1. 
\end{equation}
To see this, first note that, by \Cref{useful},
$\gamma_0\bigg(E^{B^{\prime}\cap A_1}_{D^{\prime}\cap A_1}(x)\bigg)\in
B'\cap D_1$ for all $x \in B'\cap A_1$.  Now, let $x_1\in
B^{\prime}\cap D_1$. Then, by \Cref{useful} again, there exists a $y_1\in
  D^{\prime}\cap A_1$ such that $\gamma_0(y_1)=x_1$; so
 \begin{align*}
  \tr\bigg(\gamma_0\big(E^{B^{\prime}\cap A_1}_{D^{\prime}\cap A_1}(x)\big)x_1\bigg)
  &=   \tr\bigg(\gamma_0\big(E^{B^{\prime}\cap A_1}_{D^{\prime}\cap A_1}(x)\big)\gamma_0(y_1)\bigg)\\
  &=   \tr\bigg(\gamma_0\big(y_1E^{B^{\prime}\cap A_1}_{D^{\prime}\cap A_1}(x)\big)\bigg) \hspace*{15mm}(\text{by \Cref{antiauto}})\\
    &=   \tr\bigg(y_1E^{B^{\prime}\cap A_1}_{D^{\prime}\cap A_1}(x)\bigg) \hspace*{22mm} (\text{by \Cref{gamma0istracepreserving}})\\
&= \tr \bigg(E^{B^{\prime}\cap A_1}_{D^{\prime}\cap A_1}(y_1x)\bigg)\\
&=\tr(y_1x)\\
&=\tr\big(\gamma_0(y_1x\big)) \hspace*{28mm} (\text{by \Cref{gamma0istracepreserving} again})\\
&=\tr\big(\gamma_0(x)x_1\big).
\end{align*}
 This proves \Cref{general statement}.

 Now, by \Cref{formula of p}, we have
$p(C,D)={[C:B]}_0E^{B'\cap A_1}_{C'\cap A_1}\big(e_D\big)$ and by, \Cref{relation between p
  and q}, we know that $\gamma_0\big(p(C,D)\big)=q(C,D)$.
 Thus, applying \Cref{general statement}, we obtain
\[
q(C,D)={[C:B]}_0E^{A_1}_{C_1}\big(\gamma_0(e_D)\big)=
{[C:B]}_0E^{A_1}_{C_1}(e_D),
\]
by     \Cref{gamma0ec}. The expression for $p(C,D)$ follows by symmetry. 
  \end{proof}

\subsection{A bound for the cardinality of intermediate subalgebras} \( \)

 For any unital pair $N\subset M$ of von Neumann algebras, let
 $\mathcal{I}(N\subset M)$ denote the set of its intermediate von
 Neumann subalgebras. Then, $\mathcal{I}(N\subset M)$ forms a lattice
 under the following two natural operations
 \[
 P\wedge Q:= P\cap Q ~~\text{and}~~ P\vee Q:= (P\cup Q)^{\dprime}.
 \]
 If we assume that $N\subset M$ is an irreducible subfactor (of any
 type), then $\mathcal{I}(N\subset M)$ becomes the lattice of its
 intermediate subfactors.  Watatani, in \cite{Wa2} (implicitly in
 \cite{Po3}), showed that if $N \subset M$ is a finite index
 irreducible subfactor of type $II_1$, then $\mathcal{I}(N\subset M)$
 is a finite lattice.  Subsequently, Teruya and Watatani (in
 \cite{TW}) showed that $\mathcal{I}(N\subset M)$ is  finite also if
 $N\subset M$ is a finite index irreducible subfactor of type $III$.

 On the other hand, if we consider a unital inclusion of
 $C^*$-algebras $B\subset A$, the set of intermediates
 $C^*$-subalgebras, denoted by $\mathcal{L}(B,A)$ (to distinguish it
 from the $W^*$-version), also forms a lattice under the following two
 operations:
 \[
 A\wedge B:= A\cap B ~~\text{and}~~ A\vee B:= C^*\big(A,B\big).
 \]
 Recently, Ino and Watatani (in \cite{IW}) proved that
 $\mathcal{L}(B,A)$ is finite if $A$ and $B$ are simple unital
 $C^*$-algebras with $B^{\prime}\cap A=\C$ and $[A:B]_0<\infty$. They
 did not provide any bound for the cardinality of $\mathcal{L}(A,B)$.
 Below, we provide an upper bound for the cardinality of
 $\mathcal{L}(B,A)$.  We also improve Longo's bound for the cardinality of $\mathcal{I}(N
 \subset M)$ for any finite index irreducible subfactor of type $III$.
 
 \subsubsection{Intermediate $C^*$-subalgebras of
   an irreducible pair of simple unital $C^*$-algebras}\( \)

We first observe a certain rigidity phenomenon among the minimal
intermediate $C^*$-subalgebras as was discovered for  the minimal subfactors of an
irreducible subfactor of type $II_1$  in \cite{BDLR}.
\begin{theorem}\label{m1}
Let $B \subset A$ be an irreducible inclusion of simple unital
$C^*$-algebras with a conditional expectation $E: A \rar B$ of finite
Watatani index. Then, the interior angle between any two distinct
minimal intermediate $C^*$-subalgebras $C$ and $D$ of $B \subset A$ is
greater that $\pi/3$.
\end{theorem}
\begin{proof}
First, note that, by \Cref{wata1}, $E$ is unique and hence is same as
the minimal conditional expectation $E_0$. As usual, let $E_1$ denote
the dual (also minimal) conditional expectation of $E_0$. Let $C$ and
$D$ be two distinct minimal intermediate $C^*$-subalgebras of $B \subset A$. Then, the expression
\eqref{traceformulaofangle} for interior angle yields
 \begin{align*}
 \cos \big(\alpha(C,D)\big) & =\displaystyle
 \frac{\tr(e_Ce_D)-{\big({[A:B]}_0}\big)^{-1}}{\sqrt{\big({{[A:C]}_0}\big)^{-1}-
     \big({{[A:B]}_0}\big)^{-1}}\sqrt{\big({[A:D]_0}\big)^{-1}-\big({{[A:B]}_0}\big)^{-1}}}\\ &=
 \displaystyle
 \frac{{[A:B]}_0\tr(e_Ce_D)-1}{\sqrt{{[C:B]}_0-1}\sqrt{{[D:B]}_0-1}}\\ &=
 \displaystyle \frac{t-1}{\sqrt{{[C:B]}_0-1}\sqrt{{[D:B]}_0-1}},
\end{align*}
where the last equality follows from
\Cref{pmultipleofprojection}. Also, we have
\(
\tr\big(p(C,D)\big)= r\ \left(:=
\frac{{[C:B]}_0}{{[A:D]}_0}\right),
\)
by \Cref{formulaofp2}. Thus,  from \Cref{pmultipleofprojection},  we obtain
$\tr\big([p(C,D)]\big)=\frac{r}{t}\geq \tr\big(e_C\vee e_D\big).$

Next,
recall that the projections $e_C\vee e_D-e_C$ and $e_D-e_C\wedge e_D$  are Murray von
Neumann equivalent in the finite dimensional von Neumann algebra
$B^{\prime}\cap A_1$. Therefore, 
\[
\tr(e_C\vee e_D)=
\tr(e_C)+\tr(e_D)-\tr(e_C\wedge e_D).
\]
Since $C$ and $D$ are distinct minimal intermediate $C^*$-subalgebras,
it is clear that $e_C\wedge e_D=e_B$. So, we have
$$
\frac{1}{t}\geq \frac{1}{{[C:B]}_0}+\frac{1}{{[D:B]}_0}-\frac{1}{{[C:B]}_0{[D:B]}_0}.
$$
Thus,  as  in \cite{BDLR}, we obtain
\begin{align*}
 \displaystyle \frac{t-1}{\sqrt{{[C:B]}_0-1}\sqrt{{[D:B]}_0-1}} 
 & \leq \displaystyle \frac{\sqrt{{[C:B]}_0-1}\sqrt{{[D:B]}_0-1}}{{[C:B]}_0+{[D:B]}_0-1}\\
  & < \displaystyle \frac{\sqrt{{[C:B]}_0-1}\sqrt{{[D:B]}_0-1}}{{[C:B]}_0-1+{[D:B]}_0-1}\\
 & \leq \frac{1}{2}
\end{align*}
Therefore, $\alpha(C,D)>\pi/3.$ This completes the proof.
\end{proof}

For an irreducible subfactor (of any type), Longo (in
\cite{Lon}) gave an explicit bound for the number of intermediate
subfactors by showing that the number is
bounded by $\ell^\ell$, where $\ell= [M:N]^2$. He then asked whether the
number of intermediate subfactors could be bounded by
$[M:N]^{[M:N]}$.  This question was settled for
type $II_1$ subfactors  in \cite{BDLR}  using the  notion of interior angle between
intermediate subfactors.

Now that all the necessarily tools are available to us, analogous to
the bound obtained by Longo \cite{Lon}, we first obtain a bound for
the cardinality of the lattice of intermediate $C^*$-subalgebras of an
irreducible pair $B \subset A$ of simple unital $C^*$-algebras and
then answer Longo's question for type $III$ case.

The procedure that we employ is exactly the same as was employed in \cite{BDLR}.  We provide an outline for the
reader's convenience. 
\begin{theorem}\label{bound-thm-1}
Let $B\subset A$ be an
  irreducible inclusion of simple unital $C^*$-algebras  with  a conditional expectation of finite Watatani index.
  Then, the number of intermediate $C^*$-subalgebras of $B \subset A$
  is bounded by
  \[
  \min
  \left\{9^{{[A:B]}_0^2}, \left({{[A:B]}_0^2}\right)^{{[A:B]}_0^2}\right\}.
  \]
\end{theorem}
\begin{proof}
 Let $\mathcal{L}(B,A)$ (resp., $\mathcal{L}_m(B,A)$) denote the set
of all intermediate (resp., minimal intermediate) $C^*$-subalgebras of
$B\subset A$. Then, in view of \Cref{m1}, imitating the proof of
\cite[Theorem 4.1]{BDLR}, we deduce that
\[
|\mathcal{L}_m(B,A)|\leq
3^{\text{dim}_{\C}(B^{\prime}\cap A_1)}.
\]
From \Cref{rel-comm-dimn}, we know that $\text{dim}_{\C}(B^{\prime}\cap
A_1)\leq {[A:B]}_0^2$. Thus, $|\mathcal{L}_m(B,A)|\leq 3^{{[A:B]}_0^2}.$
Now, for any ${\delta}^2\geq 2$, consider (as in \cite[Definition
  4.3]{BDLR})
\[
I(\delta^2) := \sup\left\{ |\mathcal{L}(Q,P)|: Q\subset P ~~\textrm{ is
  an irreducible inclusion of }\right.
\]
\[
\qquad \qquad \qquad  \qquad \qquad \qquad \left.  \textrm{simple unital  $C^*$-algebras with} ~~[P:Q]_0\leq
\delta^2\right\}; \text{ and }
\]
\[
m(\delta^2): = \sup \left\{ |\mathcal{L}_m(Q,P)|: Q\subset P ~~\textrm{ is
  an irreducible inclusion of}
\right.\]
  \[
\qquad \qquad \qquad  \qquad \qquad \qquad  \left.  \textrm{simple unital $C^*$-algebras with} ~~[P:Q]_0\leq
\delta^2\right\}.
\]

So, for any ${\delta}^2\geq 2$, we have $m(\delta^2)\leq
3^{\delta^4}.$ Further, since every $II_1$ factor is a simple unital
$C^*$-algebra, on the lines of \cite[Lemma 4.5]{BDLR}, we must have
$I(\delta^2)\leq m(\delta^2) I(\delta^2/2)$.

Finally, in view of \Cref{index rigidity}, proceeding as in
\cite[Theorems 4.6 and 4.7]{BDLR}, we obtain the desired bound.
\end{proof}

\subsubsection{Intermediate subfactors of an irreducible subfactor of  type $III$}\( \)

Recall that every $\sigma$-finite (equivalently, countably
decomposable) type $III$ factor is known to be simple as a
$C^*$-algebra.  Also, if $N \subset M$ is a $\sigma$-finite subfactor
of type $III$ and $P$ is an intermediate subfactor of $N \subset M$,
then $P$ is also $\sigma$-finite and of type $III$; hence, $P$ is also
a simple unital $C^*$-algebra.

Now, suppose that $N \subset M$ is a $\sigma$-finite irreducible subfactor of type
$III$ with finite Watatani index. Then, by \Cref{wata1}, it admits a
unique  (and hence minimal) conditional expectation, say, $E^M_N: M \rar N$; and
also $[M:N]_0 = \mathrm{Ind}(E^M_N)$. Clearly, $E^M_N$ is faithful
and, since $E^M_N$ satisfies the Pimsner-Popa inequality
(\Cref{pipo-inequality}), $E^M_N$ is normal as well, by
\cite[Propostion 1.1]{Po}.  So, by \cite[Proposition 2.5.3]{Wa},
$[M:N]_0$ is equal to the Kosaki index of $E^M_N$ (see \cite{kosaki}).

\begin{proposition}\label{type-III-rel-com}
  Let $N\subset M$ be an irreducible $\sigma$-finite subfactor of type
  $III$ with finite Watatani index. Then, $\mathrm{dim}(N'\cap M_1)
  \leq [M:N]_0$, where $M_1$, denotes the Watatani's $C^*$-basic
  construction for $N \subset M$ with respect to $E^M_N$.
  \end{proposition}
\begin{proof}
As observed above, $E^M_N: M \rar N$ is a faithful normal conditional
expectation with finite Kosaki index. So, by \cite{kosaki}, given any
faithful normal state $\varphi$ on $N$, there is a projection $f \in
N'\cap B(\mh)$ such that $faf = E^M_N(a) f$ for all $a \in M$, where $\mh$
is the Hilbert space $L^2(M, \varphi \circ E^M_N)$.  Then,
$\tilde{M}_1: = \text{vNa}\la M, f\ra \subseteq B(\mh)$ is called  the von
Neumann basic construction of $N \subset M$ with respect to $E^M_N$
and $\varphi$. Further, since $N$ and $M$ are of type $III$ and the Kosaki index of
  $E^M_N$ is finite, it is  known that $\text{dim}_{\C}(N^{\prime}\cap
   \tilde{M}_1)\leq {[A:B]}_0$  - see,  for example,  \cite{Sa}. Also, by \cite[Lemmas 3.2 $\&$ 3.3]{kosaki}, the mapping $M \ni a
  \mapsto af \in \tilde{M}_1$ is injective. 
Thus, by \cite[Proposition 2.2.11]{Wa} (uniqueness of $C^*$-basic
construction), there exists an injective $*$-homomorphism $\varphi:
M_1 \rar \tilde{M}_1$ such that $\varphi(e_N) = f$ and $\varphi(a) =
a$ for all $a \in M$. In particular, $\varphi(M_1) =
\overline{\text{span}}\{x f y : x, y \in M\}$ and $\varphi$ maps $N'\cap M_1$  injectively into $N'\cap \tilde{M}_1$. Hence,
\[
\text{dim}_{\C}(N' \cap M_1) \leq \text{dim}_{\C}(N^{\prime}\cap \tilde{M}_1)\leq
     {[M:N]}_0.
     \]
 This completes the proof.
\end{proof}

\begin{theorem}\label{bound-thm-2}
Let $N\subset M$ be an irreducible $\sigma$-finite subfactor of type
  $III$ with finite Watatani index. Then, the number of intermediate
  subfactors of $N \subset M$ is bounded by
  \[
  \min\bigg\{9^{{[M:N]}_0},{{[M:N]}_0}^{{[M:N]}_0}\bigg\}.
  \]
\end{theorem}

\begin{proof}
As observed above,  every
intermediate subfactor of $N \subset M$ is a simple unital
$C^*$-subalgebra.

Let $\mathcal{I}(N\subset M)$ (resp., $\mathcal{I}_m(N\subset M)$) denote the set
of all intermediate (resp., minimal intermediate) subfactors of
$N\subset M$. Then, in view of \Cref{m1}, imitating the proof of
\cite[Theorem 4.1]{BDLR}, we deduce that
\[
|\mathcal{I}_m(N\subset M)|\leq 3^{\text{dim}_{\C}(N^{\prime}\cap M_1)}.
\]
Thus, by \Cref{type-III-rel-com}, 
$|\mathcal{I}_m(N\subset M)|\leq 3^{{[M:N]}_0}.$ Now, for any ${\delta}^2\geq
2$, consider (as in \cite[Definition 4.3]{BDLR})
\[
I(\delta^2):= \sup\big\{ |\mathcal{I}(K\subset L)|:
K\subset L ~~\textrm{ is a $\sigma$-finite irreducible subfactor}
\]
\[\textrm{of type $III$  with } [L:K]_0\leq
\delta^2\big\}; \text{ and}
\]
\[
m(\delta^2):=\sup \big\{ |\mathcal{I}_m(K\subset L)|:
K\subset L ~~\textrm{ is a  $\sigma$-finite irreducible subfactor }
\]
\[
\textrm{ of type $III$ with } [L:K]_0\leq
\delta^2\big\}.
\]
So, for any ${\delta}^2\geq 2$, we have $m(\delta^2)\leq
3^{\delta^2}.$

Furthermore, there always exists a $\sigma$-finite hyperfinite factor
of type $III$ which admits an outer action of every finite group;
thus, imitating the proof of \cite[Lemma 4.5]{BDLR}, we obtain
$I(\delta^2)\leq m(\delta^2) I(\delta^2/2)$.

Finally, in view of \Cref{index rigidity}, proceeding as in
\cite[Theorems 4.6 and 4.7]{BDLR}, we obtain the desired bound.
\end{proof}

\section{Lattice of intermediate von Neumann subalgebras}\( \)

Let $\mn\subset \mm$ be a unital inclusion of von Neumann
algebras. For any such pair, as above, let $\mathcal{I}(\mn\subset
\mm)$ denote the lattice of intermediate von Neumann subalgebras.  The
main theorem of this section will show that, for a fairly large class
of such pairs, the lattice $\mathcal{I}(\mn\subset \mm)$ is always
finite. 

In order to achieve this, we will use the notion of a metric between
two subalgebras of a given $C^*$-algebra introduced by Kadison and
Kastler (in \cite{KK}) and Christensen's theory of perturbations of
operator algebras based on this metric. Recall that, if $B$ and $C$ are two
$C^*$-subalgebras of a $C^*$-algebra $A$, then the (Kadison-Kastler)
distance between $B$ and $C$ is defined as
\[
d(B, C)=\text{max} \Big\{\sup_{a\in \text{ball(B)}} \inf_{b\in
  \text{ball}(C)} \lVert a-b\rVert, \sup_{b\in \text{ball} (C)}
\inf_{a\in \text{ball}(B)} \lVert a-b\rVert\Big\}.
\]
 The following useful elementary observation is well known - see, for instance,
 \cite{Ino}.

\begin{lemma}\label{onto}
 Let $B$ and $C$ be $C^*$-subalgebras of a $C^*$-algebra $A$. If $B\subset
 C$ and $d(B,C)<1$, then $B=C$.
 \end{lemma}

\begin{notation}\label{N-in-M}
Let $\mn\subset \mm$ be a unital inclusion of finite von Neumann
algebras with a (fixed) faithful normal tracial state $\mathrm{tr}$ on
$M$.  Let $E^{\mm}_{\mn}:\mm\rightarrow \mn$ denote the unique
$\mathrm{tr}$-preserving faithful normal conditional expectation.
Also, when we restrict $\mathrm{tr}$ to $\mpp$, we obtain another
unique $\mathrm{tr}$-preserving  normal conditional
expectation $E^{\mpp}_{\mn}:\mpp \rar \mn$ and we have $
E^{\mpp}_{\mn} \circ E^{\mm}_{\mpp} = E^{\mm}_{\mn}$.
  \end{notation}

  \begin{proposition}\label{mas}
    In the set up of \Cref{N-in-M}, suppose that $E^{\mm}_{\mn}$ has
    finite Watatani index. Then, the conditional expectations
    $E^{\mm}_{\mpp}$ and $E^{\mpp}_{\mn}$ also have finite Watatani index.
\end{proposition}
\begin{proof}
 That $E^{\mpp}_{\mn}$  has finite index
 follows from \cite[Proposition 1.7.2]{Wa}. And, that
 $E^{\mm}_{\mpp}$  has finite index follows from \cite[Proposition 3.5]{M} .
\end{proof}

We now prove the main result of this section, which generalizes
  \cite[Theorem 2.2]{Wa2}. We will break the proof into two
steps. First, combining Christensen's perturbation technique from
\cite{Ch1} and an improvement by Ino \cite{Ino}, we show that if the
distance between two intermediate von Neumann subalgebras $\mpp$ and
$\mq$ is sufficiently small then they are unitarily equivalent.  Then,
following an idea of Watatani \cite{Wa2} (see also \cite{IW}), we use
a compactness argument combined with the first step to conclude that
there are only finitely many intermediate von Neumann subalgebras.

\begin{theorem}\label{main}
Let $\mn\subset \mm$ be a unital inclusion of finite von Neumann
algebras with a normal tracial state $\tr$ on $\mm$ such that the
unique $\tr$-preserving conditional $E^{\mm}_{\mn}: \mm \rar \mn$ has
finite Watatani index. If $\mn$ has finite dimensional center and
${\mn}^{\prime}\cap \mm$ equals either $\mz(\mn) $ or $\mz(\mm)$, then
the lattice $\mathcal{I}(\mn\subset \mm)$ is finite.
\end{theorem}

\begin{proof}
\noindent \textbf{Step I:} Following \cite{Ino} and \cite{Ch1}, we show
that, for every pair $\mpp, \mq \in \mathcal{I}(\mn \subset \mm)$ with
$d(\mpp,\mq)<1/15$, there exists a unitary $u$ in $\mn'\cap \mm$ such
that $u\mpp u^* = \mq$.\smallskip

 From \Cref{N-in-M} and \Cref{mas}, we see that the conditional expectations
 $E^{\mm}_{\mpp}:\mm \rar \mpp$ and $E^{\mm}_{\mq}:\mm \rar \mq$ both
 have finite index. So, they satisfy the Pimsner-Popa inequality
 (\Cref{pipo-inequality}). Thus, by \cite[Proposition 3.1]{Ino}, there
 exists a $*$-isomorphism ${\Phi}: \mq\rightarrow \mpp$ such that
 $\Phi_{|_{\mn}} = \mathrm{Id}_{\mn}$ and
 \begin{equation}
   \label{phi-id}
   \sup_{x\in \mathrm{ball}(Q)} \|\Phi(x) -
   x\| < 14\mathrm{d}(\mpp, \mq) < 1.
\end{equation}
   Then, in view \eqref{phi-id}, there exists a unitary $u\in \mm$ such
that $\Phi(x)=uxu^*$ for all $x\in \mq$, by \cite[Proposition
  4.4]{Ch1}. And, since $\Phi_{|_{\mn}} = \mathrm{Id}_{\mn}$, it
follows that $u\in {\mn}^{\prime}\cap \mm.$
\smallskip

\noindent \textbf{Step II:} We show that $\mathcal{I}(\mn\subset \mm)$ is finite.

We will again use Watatani's notion of $C^*$-basic construction. Let
$\mn \subset \mm\subset C^*\la \mm,e_{\mn}\ra$, $ \mpp\subset
\mm\subset C^*\la \mm,e_{\mpp}\ra$ and $\mq \subset \mm\subset C^*\la
\mm,e_{\mq}\ra$ denote the respective $C^*$-basic constructions with
the corresponding $C^*$-Jones projections $e_{\mn},e_{\mpp}$ and
$e_{\mq}$ respectively. Since $\mathrm{Ind}(E^{\mm}_{\mn})$ is
invertible (\cite[Lemma 2.3.1]{Wa}), the dual conditional expectation
$E^{C^*\la \mm, e_{\mm}\ra}_{\mn}: {C^*\la \mm, e_{\mn}\ra} \rar
{\mm}$ of $E^{\mm}_{\mn}$ exists and has finite index, by
\cite[Propositions 1.6.1 $\&$ 1.6.6]{Wa}; so that, $E^{\mm}_{\mm}
\circ E^{C^*\la \mm, e_{\mm}\ra}_{\mn}: C^*\la \mm, e_{\mm}\ra \rar
\mn $ also has finite index.  Thus, since $\mz(\mn)$ is finite dimensional,
the relative commutant ${\mn}^{\prime}\cap C^*\la\mm,e_{\mn}\ra$ is
finite dimensional, by \cite[ Proposition 2.7.3]{Wa}.  Hence, the
set
 $$S:= \{p\in {\mn}^{\prime}\cap C^*\la \mm,e_{\mn}\ra: p ~~\text{is a
  projection}\}$$ is a compact Hausdorff space with respect to the
norm topology.  So, for any $r>0$, there exist finitely many open
balls of diameter $r$ which cover $S$.

Fix any $0< r<\frac{1}{15\, \|\mathrm{Ind}(E^{\mm}_{\mn})\|}$. If
$e_{\mpp}$ and $e_{\mq}$ both lie in same such ball, then $\lVert
e_{\mpp}-e_{\mq}\rVert <r$; and, following \Cref{N-in-M} and
\Cref{mas}, we have $ \mathrm{IMS}(\mn, \mm, E^{\mm}_{\mn}) =
\mathcal{I}(\mn \subset \mm)$; so, by \cite[ Lemma 3.3]{IW}, we obtain
$d(\mpp,\mq)<1/15.$ Thus, by Step I, there exists a unitary $u\in
        {\mn}^{\prime}\cap \mm $ such that $u\mpp u^* = \mq$. Then,
        either ${\mn}^{\prime}\cap \mm \subseteq \mn \subseteq \mq$ or
        $\mn' \cap \mm =\mz(\mm) \subset \mm' \subset \mq'$, in both
        cases, we get $\mpp =\mq.$ Thus, there are only finitely many
        intermediate von Neumann subalgebras of the pair $\mn \subset
        \mm$. This completes the proof of the theorem.
\end{proof}

Recall that, for any unital inclusion $\mn \subset \mm $ of finite von
Neumann algebras, a representation $\pi$ of $\mm$ on a Hilbert space
$\mh$ is said to be a finite representation of the pair $\mn \subset \mm
$ if $\pi(\mn)'\subseteq B(\mh)$ is a finite von Neumann algebra. And,
the pair $\mn \subset \mm $ is said to be of finite {GHJ} index if
it admits a finite faithful representation - see \cite[$\S 3.5$]{GHJ}.

\begin{corollary}\label{GHJ}
Let $\mn\subset \mm$ be a unital inclusion of finite direct sums of
finite factors with finite $\mathrm{GHJ}$ index.  If $\mn$ has finite
dimensional center and either ${\mn}^{\prime}\cap \mm = \mz(\mn)$
or  $\mn' \cap \mm = \mz(\mm)$, then the lattice $\mathcal{I}(\mn\subset \mm)$
is finite.
\end{corollary}
\begin{proof}
 Fix a faithful normal tracial state $\mathrm{tr}$ on $M$. Then, by
 \cite[Theorem 3.6.4]{GHJ}, the unique $\tr$-preserving conditional
 expectation $E^{\mm}_{\mn}:\mm \rar \mn$ has finite Watatani
 index. The rest follows from \Cref{main}.
  \end{proof}

\begin{corollary}
 Let $\mn$ be a finite direct sum of $II_1$ factors with a finite
 group $G$ acting outerly on $\mn$. Then, the lattice $\mathcal{I}(\mn\subset
 \mn\rtimes G)$ is finite.
\end{corollary}

\begin{proof}
 Let $\mm:=\mn\rtimes G$. We know that $\mathrm{Ind}(E)=|G|$, $E$ is
 the canonical conditional expectation from $ \mn \rtimes G$ onto
 $\mn$. Further, the outerness of the action implies that
 $\mn^{\prime}\cap \mm = \mz(\mn).$ Applying \Cref{main}, we obtain the
 desired result.
\end{proof}

The following consequence can be thought of as an appropriate generalization of
\cite[Theorem 2.2]{Wa2} in the non-irreducible case.
\begin{corollary}\label{watreducible1}
 Let $N\subset M$ be a subfactor of type $II_1$ with finite Jones
 index. Then, $\mathcal{I}(\mathcal{R}\subset M)$ is a finite lattice,
 where $\mr:= N\vee (N^{\prime}\cap M).$
\end{corollary}
\begin{proof}
Since $\mr \cong N\otimes (N^{\prime}\cap M)$, it is clear that $\mr$
is a direct sum of finitely many $II_1$ factors. Then, observe that
$\mr^{\prime}\cap M\subset N^{\prime}\cap M\subset \mr.$ Thus,
$\mr^{\prime}\cap M\subset \mr\cap \mr^{\prime}$; so that
$\mathcal{R}^\prime\cap M=\mz(\mr).$ And, by \cite{M}, the $\tr_M$-preserving conditional expectation
  $E^M_{\mr}: M \rar \mr$ has finite Watatani index.  The rest again follows
from Theorem \ref{main}.
 \end{proof}

The following consequence follows from applications of the Double
Commutant Theorem and is left to the interested reader.  \color{black}
\begin{lemma}\label{lb}
Let $\mn \subset \mm$ be a pair of von Neumann algebras with common
identity. Let $\mr:= \mn \vee (\mn ^{\prime}\cap \mm)$ and $\mr_0:=
\mn \vee \mz(\mn ^{\prime}\cap \mm)$. Then, we have the following:
\begin{enumerate}
\item $\mz(\mr_0)=\mz(\mr)$.
\item  $\mn^{\prime}\cap \mr_0=\mz(\mn^{\prime}\cap \mm)=\mz(\mr_0).$
\end{enumerate}
\end{lemma}

\noindent The following implications are obvious once we apply \Cref{main}
and the preceding lemma.
\begin{corollary}
 Let $N\subset M$ be a subfactor of type $II_1$ with
 $[M:N]<\infty$. Then, the lattice $\mathcal{I}(N\subset \mr_0)$ is finite.

In particular,  if $N^{\prime}\cap M$ is abelian, then the lattices
 $\mathcal{I}(N\subset \mr)$ and $\mathcal{I}(\mr\subset \mm)$ are
 both finite.
\end{corollary}
\color{black}

\end{document}